\documentclass{amsart}%
\usepackage{amsfonts}
\usepackage{amsmath}
\usepackage{amssymb}
\usepackage{graphicx}%
\newtheorem{theorem}{Theorem}
\theoremstyle{plain}

\newtheorem{lemma}{Lemma}

\newtheorem{proposition}{Proposition}

\theoremstyle{definition}
\newtheorem{definition}{Definition}

\numberwithin{equation}{section}
\begin{document}
\title{Self-Similar Tilings of Fractal Blow-Ups}
\author[M. F. Barnsley]{M. F. Barnsley}
\address{Australian National University\\
Canberra, ACT, Australia }
\author{A. Vince}
\address{University of Florida\\
Gainesville, FL, USA}
\date{2017-06-14}

\begin{abstract}
New tilings of certain subsets of $\mathbb{R}^{M}$ are studied, tilings
associated with fractal blow-ups of certain similitude iterated function
systems (IFS). For each such IFS with attractor satisfying the open set
condition, our construction produces a usually infinite family of tilings that
satisfy the following properties: (1) the prototile set is finite; (2) the
tilings are repetitive (quasiperiodic); (3) each family contains self-similar
tilings, usually infinitely many; and (4) when the IFS is rigid in an
appropriate sense, the tiling has no non-trivial symmetry; in particular the
tiling is non-periodic.

\end{abstract}
\maketitle

\section{Introduction}

\label{sec:intro}

The subject of this paper is a new type of tiling of certain subsets $D$ of
$\mathbb{R}^{M}$. Such a domain $D$ is a fractal blow-up (as defined in
Section~\ref{defsec}) of certain similitude iterated function systems (IFSs);
see also \cite{manifold, strichartz}. For an important class of such tilings
it is the case that $D=\mathbb{R}^{M}$, as exemplified by the tiling of
Figure~\ref{fig:b} (on the right ) that is based on the \textquotedblleft
golden b" tile (on the left). We are also interested, however, in situations
where $D$ has non-integer Hausdorff dimension. The left panel in
Figure~\ref{sidebyside} shows the domain $D$, the right panel a tiling of $D$.
These examples are explored in Section~\ref{exsec}. In this work, tiles may be
fractals; pairs of distinct tiles in a tiling are required to be
non-overlapping, i.e., they intersect on a set whose Hausdorff dimension is
lower than that of the individual tiles. \begin{figure}[tbh]
\includegraphics[width=3cm, keepaspectratio]{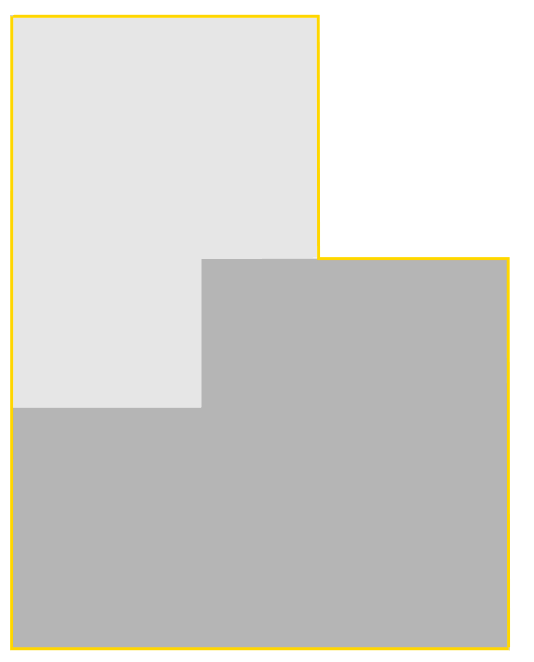} \hskip 10mm
\includegraphics[width=8cm, keepaspectratio]{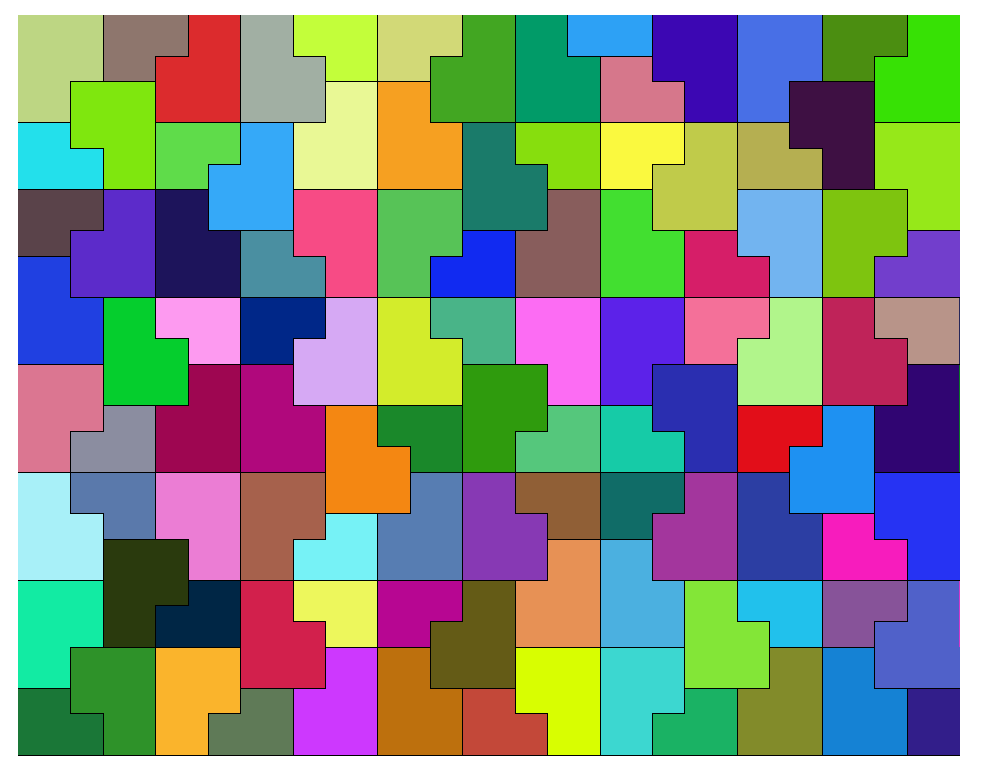} \caption{Golden b and
golden b tiling.}%
\label{fig:b}%
\end{figure}%

\begin{figure}[ptb]%
\centering
\includegraphics[
height=1.6475in,
width=5.2477in
]%
{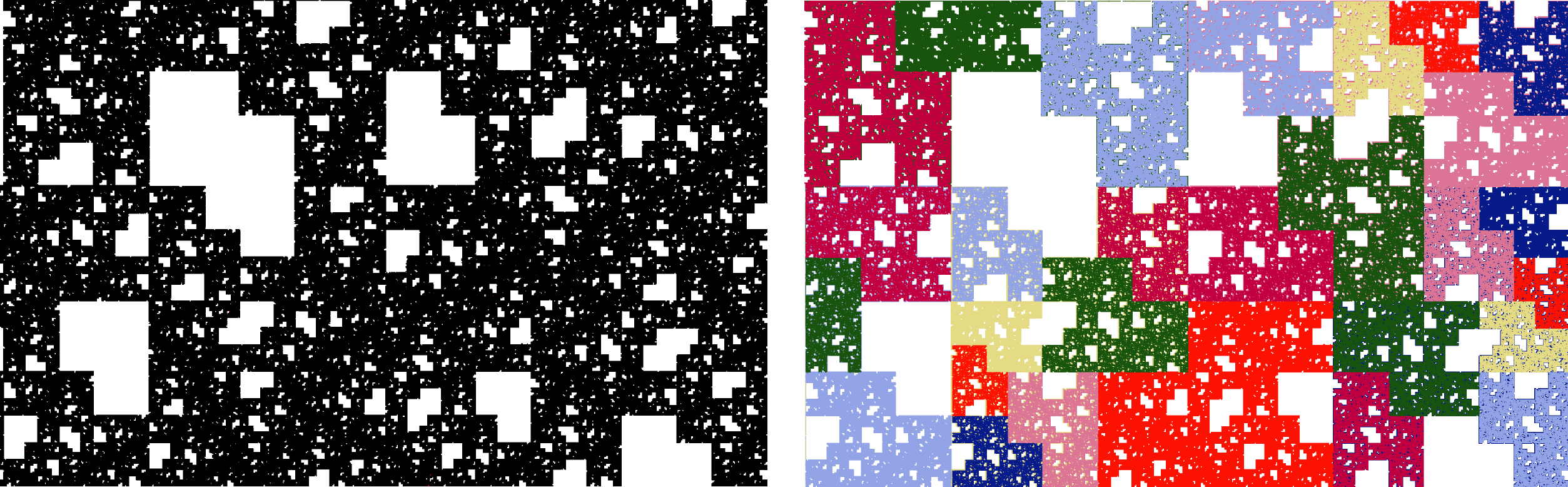}%
\caption{The left image shows part of an infinite fractal blow-up $D$; the
right image shows part of a tiling of $D$ using a finite set of prototiles.
See Section \ref{exsec}.}%
\label{sidebyside}%
\end{figure}

These tilings come in families, one family for each similitude IFS whose
functions $f_{1},f_{2}\dots,f_{N}$ have scaling ratios that are integer powers
$s^{a_{1}},s^{a_{2}},\dots,s^{a_{N}}$ of a single real number $s$ and whose
attractor is non-overlapping. Each such family contains, in general, an
uncountable number of tilings. Each family has a finite set of prototiles.

The paper is organized as follows. Sections \ref{tilingsec} and \ref{defsec}
provide background and definitions relevant to tilings and to iterated
function systems. The construction of our tilings is given in Section
\ref{defsec}. The main theorems are stated precisely in Section \ref{defsec}
and proved in subsequent sections. Results appear in Section \ref{realabsec}
that define and discuss the relative and absolute addresses of tiles. These
concepts, useful towards understanding the relationships between different
tilings, are illustrated in Section \ref{exsec}. Also in Section~\ref{exsec}
are examples of tilings of $\mathbb{R}^{2}$ and of a quadrant of
$\mathbb{R}^{2}$. The Ammann (the golden b) tilings and related fractal
tilings are also discussed in that section, as is a blow-up of a Cantor set.

A subset $P$ of a tiling $T$ is called a \textit{patch} of $T$ if it is
contained in a ball of finite radius. A tiling $T$ is \textit{quasiperiodic}
(also called repetitive) if, for any patch $P$, there is a number $R>0$ such
that any disk of radius $R$ centered at a point contained in a tile of $T$
contains an isometric copy of $P$. Two tilings are \textit{locally isomorphic}
if any patch in either tiling also appears in the other tiling. A tiling $T$
is \textit{self-similar} if there is a similitude $\psi$ such that $\psi(t)$
is a union of tiles in $T$ for all $t\in T$. Such a map $\psi$ is called a
\textit{self-similarity}.

Let $\mathcal{F}$ be a similitude IFS whose functions have scaling ratios
$s^{a_{1}},s^{a_{2}},\dots,s^{a_{N}}$ as defined above. Let $[N]^{\ast}$ be
the set of finite words over the alphabet $[N]:=\{1,2,\dots,N\}$ and
$[N]^{\infty}$ be the set of infinite words over the alphabet $[N]$. For a
fixed IFS $\mathcal{F}$, our results show that:

\begin{enumerate}
\item For each $\theta\in[N]^{*}$, our construction yields a bounded tiling,
and for each $\theta\in[N]^{\infty}$, our construction yields an unbounded
tiling. In the latter case, the tiling, denoted $\pi(\theta)$, almost always
covers ${\mathbb{R}}^{M}$ when the attractor of the IFS has nonempty interior.

\item The mapping $\theta\mapsto\pi(\theta)$ is continuous with respect to the
standard topologies on the domain and range of $\pi$.

\item Under quite general conditions, the mapping $\theta\mapsto\pi(\theta)$
is injective.

\item For each such tiling, the prototile set is $\{sA, s^{2}A,\dots,
s^{a_{\max}} A\}$, where $A$ is the attractor of the IFS and $a_{\max} =
\max\{a_{1}, a_{2}, \dots, a_{N}\}$.

\item The constructed tilings, in the unbounded case, are repetitive
(quasiperiodic) and any two such tilings are locally isomorphic.

\item For all $\theta\in[N]^{\infty}$, if $\theta$ is eventually periodic,
then $\pi(\theta)$ is self-similar.

\item If $\mathcal{F}$ is strongly rigid, then how isometric copies of a pair
bounded tilings can overlap is extremely restricted: if the two tilings are
such that their overlap is a subset of each, then one tiling must be contained
in the other.

\item If $\mathcal{F}$ is strongly rigid, then the constructed tilings have no
non-identity symmetry. In particular, they are non-periodic.
\end{enumerate}

The concept of a rigid and a strongly rigid IFS is discussed in
Sections~\ref{strongsec}.

A special case of our construction (polygonal tilings, no fractals) appears in
\cite{polygon}, in which we took a more recreational approach, devoid of
proofs. Other references to related material are \cite{anderson, sadun}. This
work extends, but is markedly different from \cite{tilings}.

\section{\label{tilingsec}Tilings, Similitudes and Tiling Spaces}

Given a natural number $M$, this paper is concerned with certain tilings of
strict subsets of Euclidean space $\mathbb{R}^{M}$ and of $\mathbb{R}^{M}$
itself. A \textit{tile} is a perfect (i.e. no isolated points) compact
nonempty subset of $\mathbb{R}^{M}$. Fix a Hausdorff dimension $0<D_{H}\leq
M$. A \textit{tiling} in $\mathbb{R}^{M}$ is a set of tiles, each of Hausdorff
dimension $D_{H}$, such that every distinct pair is non-overlapping. Two tiles
are \textit{non-overlapping} if their intersection is of Hausdorff dimension
strictly less than $D_{H}$. The \textit{support} of a tiling is the union of
its tiles. We say that a tiling tiles its support. Some examples are presented
in Section~\ref{exsec}.

A \textit{similitude} is an affine transformation $f:{\mathbb{R}}%
^{M}\rightarrow{\mathbb{R}}^{M}$ of the form $f(x)=s\,O(x)+q$, where $O$ is an
orthogonal transformation and $q\in\mathbb{R}^{M}$ is the translational part
of $f(x)$. The real number $s>0$, a measure of the expansion or contraction of
the similitude, is called its \textit{scaling} \textit{ratio}. An
\textit{isometry} is a similitude of unit scaling ratio and we say that two
sets are isometric if they are related by an isometry. We write $\mathcal{E}$
to denote the group of isometries on $\mathbb{R}^{M}$.

The \textit{prototile set} $\mathcal{P}$ of a tiling $T$ is a minimal set of
tiles such that every tile in $T$ is an isometric copy of a tile in
$\mathcal{P}$. The tilings constructed in this paper have a finite prototile set.

Given a tiling $T$ we define $\partial T$ to be the union of the set of
boundaries of all of the tiles in $T$ and we let $\rho:\mathbb{R}%
^{M}\rightarrow\mathbb{S}^{M}$ be the usual $M$-dimensional stereographic
projection to the $M$-sphere, obtained by positioning $\mathbb{S}^{M}$ tangent
to $\mathbb{R}^{M}$ at the origin. We define the distance between tilings $T$
and $T^{\prime}$ to be%
\[
d_{\tau}(T,T^{\prime})=h(\overline{\rho(\partial T)},\overline{\rho(\partial
T^{\prime})})
\]
where the bar denotes closure and $h$ is the Hausdorff distance with respect
to the round metric on $\mathbb{S}^{M}$. Let $\mathbb{K(R}^{M})$ be the set of
nonempty compact subsets of $\mathbb{R}^{M}$. It is well known that $d_{\tau}$
provides a metric on the space $\mathbb{K}({\mathbb{R}}^{M})$ and that
$(\mathbb{K}({\mathbb{R}}^{M}),d_{\tau})$ is a compact metric space.

This paper examines spaces consisting, for example, of $\pi(\theta)$ indexed
by $\theta\in\left[  N\right]  ^{\ast}$ with metric $d_{\tau}$. Although we
are aware of the large literature on tiling spaces, we do not explore the
larger spaces obtained by taking the closure of orbits of our tilings under
groups of isometries as in, for example, \cite{anderson, sadun}. We focus on
the relationship between the addressing structures associated with IFS theory
and the particular families of tilings constructed here.

\section{\label{defsec} Definition and Properties of IFS Tilings}

Let $\mathbb{N=\{}1,2,\cdots\}$ and $\mathbb{N}_{0}=\{0,1,2,\cdots\}$. For
$N\in\mathbb{N}$, let $[N]=\{1,2,\cdots,N\}$. Let $[N]^{\ast}=\cup
_{k\in\mathbb{N}_{0}}[N]^{k}$, where $[N]^{0}$ is the empty string, denoted
$\varnothing$.

See \cite{hutchinson} for formal background on iterated function systems
(IFSs). Here we are concerned with IFSs of a special form: let $\mathcal{F}%
=\{{\mathbb{R}}^{M};f_{1},f_{2},\cdots,f_{N}\}$, with $N\geq2$, be an IFS of
contractive similitudes where the scaling factor of $f_{n}$ is $s^{a_{n}}$
with $0<s<1$ where $a_{n}\in\mathbb{N}$. There is no loss of generality in
assuming that the greatest common divisor is one: $\gcd\{a_{1},a_{2}%
,\cdots,a_{N}\}=1$. That is, for $x\in{\mathbb{R}}^{M}$, the function
$f_{n}:\mathbb{R}^{M}\rightarrow\mathbb{R}^{M}$ is defined by
\[
f_{n}(x)=s^{a_{n}}O_{n}(x)+q_{n}%
\]
where $O_{n}$ is an orthogonal transformation and $q_{n}\in{\mathbb{R}}^{M}$.
It is convenient to define%
\[
a_{\max}=\max\{a_{i}:i=1,2,\dots,N\}.
\]

The \textit{attractor} $A$ of $\mathcal{F}$ is the unique solution in
$\mathbb{K(R}^{M})$ to the equation
\[
A=\bigcup\limits_{i\in\lbrack N]}f_{i}(A)\text{.}%
\]
It is assumed throughout that $A$ obeys the open set condition (OSC) with
respect to $\mathcal{F}$. As a consequence, the intersection of each pair of
distinct tiles in the tilings that we construct either have empty intersection
or intersect on a relatively small set. More precisely, the OSC implies that
the Hausdorff dimension of $A$ is strictly greater than the Hausdorff
dimension of the \textit{set of overlap} $\mathcal{O=\cup}_{i\neq j}%
f_{i}(A)\cap f_{j}(A)$. Similitudes applied to subsets of the set of overlap
comprise the sets of points at which tiles may meet. See \cite[p.481]{bandt}
for a discussion concerning measures of attractors compared to measures of the
set of overlap.


In what follows, the space $[N]^{\ast}\cup\lbrack N]^{\infty}$ is equipped
with a metric $d_{[N]^{\ast}\cup\lbrack N]^{\infty}}$ such that it becomes
compact. First, define the \textquotedblleft length" $\left\vert
\theta\right\vert $ of $\theta\in\lbrack N]^{\ast}\cup\lbrack N]^{\infty}$ as
follows. For $\theta=\theta_{1}\theta_{2}\cdots\theta_{k}\in\lbrack N]^{\ast}$
define $\left\vert \theta\right\vert =k$, and for $\theta\in\lbrack
N]^{\infty}$ define $\left\vert \theta\right\vert =\infty$. Now define
$d_{[N]^{\ast}\cup\lbrack N]^{\infty}}(\theta,\omega)=0$ if $\theta=\omega,$
and
\[
d_{[N]^{\ast}\cup\lbrack N]^{\infty}}(\theta,\omega)=2^{-\mathcal{N(}%
\theta,\omega)}%
\]
if $\theta\neq\omega$, where $\mathcal{N(}\theta,\omega)$ is the index of the
first disagreement between $\theta$ and $\omega$ (and $\theta$ and $\omega$
are understood to disagree at index $k$ if either $|\theta|<k$ or $|\omega|<k$
). It is routine to prove that $([N]^{\ast}\cup\lbrack N]^{\infty
},d_{[N]^{\ast}\cup\lbrack N]^{\infty}})$ is a compact metric space.

A point $\theta\in\lbrack N]^{\infty}$ is \textit{eventually periodic} if
there exists $m\in\mathbb{N}_{0}$ and $n\in\mathbb{N}$ such that $\theta
_{m+i}=\theta_{m+n+i}$ for all $i\geq1$. In this case we write $\theta
=\theta_{1}\theta_{2}\cdots\theta_{m}\overline{\theta_{m+1}\theta_{m+2}%
\cdots\theta_{m+n}}$.

For $\theta=\theta_{1}\theta_{2}\cdots\theta_{k}\in\lbrack N]^{\ast}$, the
following simplifying notation will be used:
\[
\begin{aligned} f_{\theta} &= f_{\theta_{1}} f_{\theta_{2}}\cdots f_{\theta_k} \\
f_{-\theta} &=f_{\theta_{1}}^{-1}f_{\theta_{2}}^{-1}\cdots f_{\theta_k}^{-1}=(f_{\theta_k\theta_{k-1}\cdots\theta_{1}})^{-1},
\end{aligned}
\]
with the convention that $f_{\theta}$ and $f_{-\theta}$ are the identity
function $id$ if $\theta=\varnothing$. Likewise, for all $\theta\in\lbrack
N]^{\infty}$ and $k\in\mathbb{N}_{0}$ define $\theta|k=\theta_{1}\theta
_{2}\cdots\theta_{k}$, and
\[
f_{-\theta|k}=f_{\theta_{1}}^{-1}f_{\theta_{2}}^{-1}\cdots f_{\theta_{k}}%
^{-1}=(f_{\theta_{k}\theta_{k-1}\cdots\theta_{1}})^{-1},
\]
with the convention that $f_{-\theta|0}=id$.

For $\sigma=\sigma_{1}\sigma_{2}\cdots\sigma_{k}\in\lbrack N]^{\ast}$ and with
$\left\{  a_{1},\dots,a_{N}\right\}  $ the scaling powers defined above, let
\[
e(\sigma)=a_{\sigma_{1}}+a_{\sigma_{2}}+\cdots+a_{\sigma_{k}}\qquad
\text{and}\qquad e^{-}(\sigma)=a_{\sigma_{1}}+a_{\sigma_{2}}+\cdots
+a_{\sigma_{k-1}},
\]
with the conventions $e(\varnothing)=e^{-}(\varnothing)=0$. Let
\[
\Omega_{k}:=\{\sigma\in\lbrack N]^{\ast}:e(\sigma)>k\geq e^{-}(\sigma)\}
\]
for all $k\in\mathbb{N}_{0}$, and note that $\Omega_{0}=[N]$. We also write,
in some places, $\sigma^{-}=\sigma_{1}\sigma_{2}\cdots\sigma_{k-1}$ so that
\[
e^{-}(\sigma)=e(\sigma^{-}).
\]

\begin{definition}
\label{defONE} A mapping $\mathbb{\pi}$ from $[N]^{\ast}\cup\lbrack
N]^{\infty}$ to collections of subsets of $\mathbb{R}^{M}$ is defined as
follows. For $\theta\in\lbrack N]^{\ast}$
\[
\mathbb{\pi}(\theta):=\{f_{_{-\theta}}f_{\sigma}(A):\sigma\in\Omega
_{e(\theta)}\},
\]
and for $\theta\in\lbrack N]^{\infty}$
\[
\mathbb{\pi}(\theta):=\bigcup\limits_{k\in\mathbb{N}_{0}} \pi(\theta|k).
\]


Let $\mathbb{T}$ be the image of $\pi$, i.e.
\[
\mathbb{T=\{\pi}(\theta):\theta\in\lbrack N]^{\ast}\cup\lbrack N]^{\infty}\}.
\]

\end{definition}

It is consequence of Theorem~\ref{theoremONE}, stated below, that the elements
of $\mathbb{T}$ are tilings. We refer to $\mathbb{\pi}(\theta)$ as an
\textit{IFS tiling}, but usually drop the term \textquotedblleft IFS". It is a
consequence of the proof of Theorem \ref{theoremONE}, given in Section
\ref{ProofofONE}, that the support of $\mathbb{\pi}(\theta)$ is what is
sometimes referred to as a \textit{fractal blow-up} \cite{manifold,
strichartz}. More exactly, if $F_{k}:=f_{_{-\theta|k}}(A)$, then
\[
\text{support}\,(\mathbb{\pi}(\theta))=\bigcup\limits_{k\in\mathbb{N}_{0}%
}F_{k}.
\]
Thus the support of $\mathbb{\pi}(\theta)$ is the limit of an increasing union
of sets $F_{0}\subseteq F_{1}\subseteq F_{2}\subseteq\cdots$, each similar to
$A$.

The theorems of this paper are summarized in the rest of this section. The
first two theorems, as well as a proposition in Section \ref{realabsec},
reveal general information about the tilings in $\mathbb{T}$ without the
rigidity condition that is assumed in the second two theorems. The proof of
the following theorem appears in Section~\ref{ProofofONE}.

\begin{theorem}
\label{theoremONE} Each set $\pi(\theta)$ in $\mathbb{T}$ is a tiling of a
subset of $\mathbb{R}^{M}$, the subset being bounded when $\theta\in[N]^{*}$
and unbounded when $\theta\in\lbrack N]^{\infty}$. For all $\theta\in\lbrack
N]^{\infty}$ the sequence of tilings $\left\{  \pi(\theta|k)\right\}
_{k=0}^{\infty}$ is nested according to%
\begin{equation}
\{f_{i}(A):i\in\lbrack N]\}=\pi(\varnothing)\subset\pi(\theta|1)\subset
\pi(\theta|2)\subset\pi(\theta|3)\subset\cdots\text{ .} \label{eqthmONE}%
\end{equation}
For all $\theta\in\lbrack N]^{\infty}$, the prototile set for $\mathbb{\pi
}(\theta)$ is $\{s^{i}A:i=1,2,\cdots,a_{\max}\}$.\textit{ }Furthermore
\[
\pi:[N]^{\ast}\cup\lbrack N]^{\infty}\rightarrow\mathbb{T}%
\]
is a continuous map from the compact metric space $[N]^{\ast}\cup\lbrack
N]^{\infty}$ into the space $(\mathbb{K}({\mathbb{R}}^{M}),d_{\tau})$.
\end{theorem}

The proof of the following theorem is given in Section \ref{proofofTHREE}.

\begin{theorem}
\label{theoremTHREE}

\begin{enumerate}
\item Each tiling in $\mathbb{T}$ is quasiperiodic and each pair of such
tilings in $\mathbb{T}$ are locally isomorphic.

\item If $\theta$ is eventually periodic, then $\pi(\theta)$ is self-similar.
In fact, if $\theta=\alpha\overline{\beta}$ for some $\alpha,\beta\in\left[
N\right]  ^{\ast}$ then $f_{-\alpha}f_{-\beta}\left(  f_{-\alpha}\right)
^{-1}$ is a self-similarity of $\pi(\theta)$.
\end{enumerate}
\end{theorem}

In Section \ref{strongsec} the concept of \textit{rigidity} of an IFS is
defined. We postpone the definition because additional notation is required.
There are numerous examples of rigid $\mathcal{F}$, including the golden b IFS
in Section \ref{exsec}. The following theorem is proved in Section
\ref{strongsec}.

\begin{theorem}
\label{intersectthm} Let $\mathcal{F}$ be strongly rigid. If $\theta
,\theta^{\prime}\in\lbrack N]^{\ast}$and $E\in\mathcal{E}$ are such that
$\pi(\theta)\cap E\pi(\theta^{\prime})$ is a nonempty common tiling, then
either $\pi(\theta)\subset E\pi(\theta^{\prime})$ or $E\pi(\theta^{\prime
})\subset\pi(\theta)$. If $e(\theta)=e(\theta^{\prime}),$ then $E\pi
(\theta^{\prime})=\pi(\theta).$
\end{theorem}

A \textit{symmetry} of a tiling is an isometry that takes tiles to tiles. A
tiling is \textit{periodic} if there exists a translational symmetry;
otherwise the tiling is \textit{non-periodic}. For example, any tiling of a
quadrant of $\mathbb{R}^{2}$ by congruent squares is periodic. The proof of
the following theorem is given in Section \ref{proofofTWO}.

\begin{theorem}
\label{theoremTWO}If $\mathcal{F}$ is strongly rigid, then there does not
exist any non-identity isometry $E\in\mathcal{E}$ and $\theta\in\lbrack
N]^{\infty}$ such that $E\pi(\theta)\subset\pi(\theta)$.
\end{theorem}

The following theorem is proved in Section~\ref{invertsec}.

\begin{theorem}
\label{1to1thm}If $\pi(i)\cap\pi(j)$ does not tile $\left(  support\text{ }%
\pi(i)\right)  \cap\left(  support\text{ }\pi(j)\right)  $ for all $i\neq j$,
then $\pi:[N]^{\ast}\cup\lbrack N]^{\infty}\rightarrow\mathbb{T}$ is one-to-one.
\end{theorem}

\section{Structure of $\{\Omega_{k}\}$ and Symbolic IFS Tilings}

The results in this section, which will be applied later, relate to a symbolic
version of the theory in this paper. The next two lemmas provide recursions
for the sequence $\Omega_{k}:=\{\sigma\in\lbrack N]^{\ast}:e(\sigma)>k\geq
e^{-}(\sigma)\}$. In this section the square union symbol $\bigsqcup$ denotes
a disjoint union.

\begin{lemma}
\label{lemma1} For all $k\geq a_{\max}$
\begin{equation}
\Omega_{k}={\bigsqcup_{i=1}^{N}i\, \Omega_{k-a_{i}}}. \label{indexformula}%
\end{equation}

\end{lemma}

\begin{proof}
For all $k\in\mathbb{N}_{0}$ we have%
\begin{align*}
i\,\Omega_{k}  &  =\{i\sigma:\sigma\in\lbrack N]^{\ast},e(\sigma) > k \geq
e^{-}(\sigma)\}\\
&  =\{\omega:\omega\in\lbrack N]^{\ast},e(\omega) > k+a_{i} \geq e^{-}%
(\omega),\omega_{1}=i\}\\
&  =\Omega_{k+a_{i}}\cap i[N]^{\ast}\text{.}%
\end{align*}
It follows that
\[
i\,\Omega_{k-a_{i}}=\Omega_{k}\cap i[N]^{\ast}%
\]
for all $k\geq a_{i}$, from which it follows that $\Omega_{k}={\bigsqcup
_{i=1}^{N}i\Omega_{k-a_{i}}}$ for all $k\geq a_{\max}$.
\end{proof}

\begin{lemma}
\label{lemstruc2} With $\Omega_{k}^{^{\prime}} :=\{\omega\in\lbrack N]^{\ast
}:e(\omega)=k+1\}$, we have $\Omega_{k}^{^{\prime}}\subset\Omega_{k}$ and
\[
\Omega_{k+1}=\{\Omega_{k}\backslash\Omega_{k}^{\prime}\} \, \bigsqcup\,
\left\{  {\bigsqcup_{i=1}^{N}\Omega_{k}^{^{\prime}}i}\right\}  .
\]

\end{lemma}

\begin{proof}
(i) We first show that $\{\Omega_{k}\backslash\Omega_{k}^{\prime}\} \,
\bigsqcup\, \left\{  {\bigsqcup_{i=1}^{N}\Omega_{k}^{^{\prime}}i}\right\}
\subset\Omega_{k+1}$.

Suppose $\theta\in\Omega_{k}\backslash\Omega_{k}^{\prime}$. Then $e^{-}%
(\theta)\leq k<e(\theta)$ and $e(\theta)\neq k+1$. Hence $e^{-}(\theta)\leq
k+1<e(\theta)$ and so $\theta\in\Omega_{k+1}$.

Suppose ${\theta\in\Omega_{k}^{^{\prime}}i}$ for some $i\in\lbrack N].$ Then
${\theta=\theta}^{-}{i}$ where ${\theta}^{-}\in{\Omega_{k}^{^{\prime}}}$,
$e^{-}(\theta)=e({\theta}^{-})=k+1$ and $e(\theta)=e({\theta}^{-}%
{i)=k+1+a}_{i}.$ Hence $e\left(  \theta\right)  >k+1=e^{-}(\theta)$. Hence
$e^{-}(\theta)\leq k+1<e\left(  \theta\right)  $. Hence $\theta\in\Omega
_{k+1}$.

(ii) We next show that $\Omega_{k+1}\subset\{\Omega_{k}\backslash\Omega
_{k}^{\prime}\}\, \bigsqcup\, \left\{  {\bigsqcup_{i=1}^{N}\Omega
_{k}^{^{\prime}}i}\right\}  $.

Let $\theta\in\Omega_{k+1}.$ Then $e^{-}(\theta)=e(\theta^{-})\leq
k+1<e(\theta)$.

If $e(\theta^{-})=k+1$, then $\theta\in{\Omega_{k}^{^{\prime}}}\theta
_{\left\vert \theta\right\vert }\subset\{\Omega_{k}\backslash\Omega
_{k}^{\prime}\} \, \bigsqcup\, \left\{  {\bigsqcup_{i=1}^{N}\Omega
_{k}^{^{\prime}}i}\right\}  $.

If $e(\theta^{-})\neq k+1$, then $e(\theta^{-})<k+1$. So $e(\theta^{-})\leq
k<k+1<e(\theta)$; so $\theta\in\Omega_{k}\backslash\Omega_{k}^{\prime}%
\subset\{\Omega_{k}\backslash\Omega_{k}^{\prime}\} \, \bigsqcup\, \left\{
{\bigsqcup_{i=1}^{N}\Omega_{k}^{^{\prime}}i}\right\}  $.
\end{proof}

For all $\theta\in\lbrack N]^{\ast},$ define $c(\theta)=\{\omega\in\lbrack
N]^{\infty}:\omega_{1}\omega_{2}\cdots\omega_{\left\vert \theta\right\vert
}=\theta\}$. (Such sets are sometimes called \textit{cylinder sets}.) With the
metric on $[N]^{\infty}$ defined to be $d_{0}(\theta,\omega)=2^{-\min
\{k:\theta_{k}\neq\omega_{k}\}}$ for $\theta\neq\omega$, the diameter of
$c(\theta)$ is $2^{-(\left\vert \theta\right\vert +1)}$. The following lemma
tells us how $\{c(\theta):\theta\in\Omega_{k}\}$ may be considered as a tiling
of the symbolic space $[N]^{\infty}$.

\begin{lemma}
\label{lemstruc3} For each $k\in\mathbb{N}_{0}$ the collection of sets
$\{c(\theta):\theta\in\Omega_{k}\}$ form a partition of $[N]^{\infty}$, each
part of which has diameter belonging to $\{s^{k+1},s^{k+2},\dots s^{k+a_{\max
}}\}$ where $s=1/2$. That is,
\[
\left[  N\right]  ^{\infty}=\bigsqcup\limits_{\theta\in\Omega_{k}}c(\theta)
\]
for all $k\in\mathbb{N}_{0}$.
\end{lemma}

\begin{proof}
Assume that $\omega\in[N]^{\infty}$. There is a unique $j$ such that $\omega|j
\in\Omega_{k}$. Letting $\theta= w|j$ we have $\omega\in c(\theta)
\subset[N]^{\infty}$. Therefore $[N]^{\infty} = \bigcup_{\theta\in\Omega_{k}}
c(\theta)$.

Assume that $\theta,\theta^{\prime}\in\Omega_{k}$. If $\omega\in c(\theta)\cap
c(\theta^{\prime})$, then by the definition of cylinder set either
$\theta=\theta^{\prime}$ or $|\theta|\neq|\theta^{\prime}|$. However, if
$|\theta|\neq|\theta^{\prime}|$, then $\omega\big ||\theta|=\theta\in
\Omega_{k}$ and $\omega\big ||\theta^{\prime}|=\theta^{\prime}\in\Omega_{k}$,
which would contradict the uniqueness of $j$. Therefore $[N]^{\infty
}=\bigsqcup_{\theta\in\Omega_{k}}c(\theta)$.
\end{proof}


\section{A Canonical Sequence of Self-similar Tilings}

To facilitate the proofs of the theorems stated in Section~\ref{defsec},
another family of tilings is introduced, tilings isometric to those that are
the subject of this paper. Let
\[
A_{k}=s^{-k}A
\]
for all $k\in\mathbb{N\cup\{}-1,-2,\dots,-a_{\max}\}$, and define, for all
$k\in\mathbb{N}$, a sequence of tilings $T_{k}$ of $A_{k}$ by
\[
T_{k}=\{ s^{-k} f_{\sigma}(A):\sigma\in\Omega_{k}\}.
\]

The following lemma says, in particular, that $T_{k}$ is a non-overlapping
union of copies of $T_{k-a_{i}}$ for $i\in\lbrack N]$ when $k\geq a_{\max}$,
and $T_{k}$ may be expressed as a non-overlapping union of copies of
$T_{k-e(\omega)}$ for $\omega\in\Omega_{{l}}$ when $k$ is somewhat larger than
$l\in\mathbb{N}_{0}$. In this section the square union notation $\bigsqcup$
denotes a non-overlapping union.

\begin{lemma}
\label{lemma02} For all $k\in\mathbb{N}_{0}$ the support of $T_{k}$ is $A_{k}%
$. For all $\theta\in\lbrack N]^{\ast}$,
\[
\pi(\theta)=E_{\theta}T_{e(\theta)}%
\]
where $E_{\theta}$ is the isometry $f_{-\theta}s^{e(\theta)}$. Also
\begin{equation}
T_{k}={\bigsqcup_{i=1}^{N}}E_{k,i}T_{k-a_{i}} \label{Tkformula}%
\end{equation}
for all $k\geq a_{\max}$, where each of the mappings $E_{k,i}=s^{-k}\circ
f_{i}\circ s^{k-a_{i}}$ is an isometry. More generally,
\begin{equation}
T_{k}={\bigsqcup_{\omega\in\Omega_{l}}}E_{k,\omega}T_{k-e(\omega)},
\label{Tkformula2}%
\end{equation}
for all $k\geq l+a_{\max}$ and for all $l\in\mathbb{N}_{0}$, where each of the
mappings $E_{k,\omega} =s^{-k}\circ f_{\omega}\circ s^{k-e(\omega)}$ is an isometry.
\end{lemma}

\begin{proof}
It is well-known that if $\mathcal{P}$ is a partition of $[N]^{\infty},$ then
$A=%
{\textstyle\bigcup_{\omega\in\mathcal{P}}}
\phi(\omega)$ where $\phi:[N]^{\infty}\rightarrow A$ is the usual (continuous)
coding map defined by $\phi(\omega)=\lim_{k\rightarrow\infty}f_{\omega|k}(x)$
for any fixed $x\in A$. By Lemma \ref{lemstruc3} we can choose $\mathcal{P=}%
\{c(\theta):\theta\in\Omega_{k}\}$. Hence, the support of $T_{k}$ is
\begin{align*}
s^{-k}\{%
{\textstyle\bigcup}
\{f_{\sigma}(A) :\sigma\in\Omega_{k}\}\}  &  =s^{-k}\{%
{\textstyle\bigcup}
\{\phi(\omega):\omega\in\{c(\theta):\theta\in\Omega_{k}\}\}\}\\
&  =s^{-k}A\text{.}%
\end{align*}

The expression $\pi(\theta)=E_{\theta}T_{e(\theta)}$ where $E_{\theta}$ is the
isometry $f_{-\theta}s^{e(\theta)}$ follows from the definitions of
$\pi(\theta)$ and $T_{k}$ on taking $k=e(\theta)$.

Equation (\ref{Tkformula}) follows from Lemma \ref{lemma1} according to these
steps.
\begin{align*}
T_{k}  &  = \{ s^{-k} f_{\sigma}(A):\sigma\in\Omega_{k}\}\text{ (by
definition)}\\
&  =s^{-k}\{f_{\sigma}(A):\sigma\in{\bigsqcup_{i=1}^{N}i\Omega_{k-a_{i}}%
\}}\text{ (by Lemma \ref{lemma1})}\\
&  =s^{-k}{\bigsqcup_{i=1}^{N}}\{f_{i\sigma}(A):\sigma\in\Omega{_{k-a_{i}}%
\}}\text{ (identity)}\\
&  =s^{-k}{\bigsqcup_{i=1}^{N}}f_{i}(\{f_{\sigma}(A):\sigma\in\Omega
{_{k-a_{i}}\})}\text{ (identity)}\\
&  ={\bigsqcup_{i=1}^{N}}E_{k,i}T_{k-a_{i}}\text{ (by definition)}%
\end{align*}

The function $E_{k,i}=s^{-k}\circ f_{i}\circ s^{k-a_{i}}$ is an isometry
because it is a composition of three similitudes, of scaling ratios $s^{-k}$,
$s^{a_{i}},$ and $s^{k-a_{i}}$. The proof of the last assertion is immediate:
tiles meet at images under similitudes of the set of overlap $\mathcal{O=\cup
}_{i\neq j}f_{i}(A)\cap f_{j}(A)$.

Equation (\ref{Tkformula2}) can be proved by induction on $l,$ starting from
Equation (\ref{Tkformula}) and using Lemma \ref{lemstruc2}.
\end{proof}

The following definition, formalizing the notion of an ``isometric combination
of tilings", will be used later, but it is convenient to place it here.

\begin{definition}
Let $\{U_{i}:i\in\mathcal{I\}}$ be a collection of tilings. An
\textbf{isometric combination of the set of tilings} $\{U_{i}:i\in
\mathcal{I\}}$ is a tiling $V$ that can be written in the form
\[
V={\bigsqcup_{i=1}^{K}}E^{(i)}U^{(i)}%
\]
for some $K\in\mathbb{N}$, where $E^{(i)}\in\mathcal{E}$, $U^{(i)}\in
\{U_{i}:i\in\mathcal{I\}}$, for all $i\in\{1,2,\dots,K\}.$
\end{definition}

For example, Lemma \ref{lemma02} tells us that any $T_{k}$ can be written as
an isometric combination of any set of tilings of the form $\{T_{j,}%
T_{j+1},\dots,T_{j+a_{\max}-1}\}$ when $k\geqslant j.$

\begin{proposition}
\label{lemmass} The sequence $\left\{  T_{k}\right\}  $ of tilings is
self-similar in the following sense. Each of the sets in the magnified tiling
$s^{-1}T_{k}$ is a union of tiles in $T_{k+1}$.
\end{proposition}

\begin{proof}
This follows at once from Lemma \ref{lemstruc2}. The tiling $T_{k+1}$ is
obtained from $T_{k}$ by applying the similitude $s^{-1}$ and then splitting
those resulting sets that are isometric to $A$. By splitting we mean we
replace $EA$ by $\{Ef_{1}(A),$ $Ef_{2}(A),\dots, Ef_{N}(A)\}$, see Section
\ref{strongsec}.
\end{proof}

\section{Theorem \ref{theoremONE}: Existence and Continuity of
Tilings\label{ProofofONE}}

Let
\[
A_{-\theta|k}:=f_{-\theta|k}A
\]
for all $\theta\in\lbrack N]^{\infty}$. It is immediate from Definition
\ref{defONE} that the support of the tiling $\pi(\theta|k)$ is $A_{-\theta|k}$
and that $\pi(\theta|k)$ is isometric to the tiling $T_{e(k)}$ of $A_{e(k)}$.
We use this fact repeatedly in the rest of this paper.

\begin{flushleft}
\textbf{Theorem~\ref{theoremONE}.} Each set $\pi(\theta)$ in $\mathbb{T}$ is a
tiling of a subset of $\mathbb{R}^{M}$, the subset being bounded when
$\theta\in[N]^{*}$ and unbounded when $\theta\in\lbrack N]^{\infty}$. For all
$\theta\in\lbrack N]^{\infty}$ the sequence of tilings $\left\{  \pi
(\theta|k)\right\}  _{k=0}^{\infty}$ is nested according to%
\begin{equation}
\{f_{i}(A):i\in\lbrack N]\}=\pi(\varnothing)\subset\pi(\theta|1)\subset
\pi(\theta|2)\subset\pi(\theta|3)\subset\cdots\text{ .}%
\end{equation}
For all $\theta\in\lbrack N]^{\infty}$, the prototile set for $\mathbb{\pi
}(\theta)$ is $\{s^{i}A:i=1,2,\cdots,a_{\max}\}$.\textit{ }Furthermore
\[
\pi:[N]^{\ast}\cup\lbrack N]^{\infty}\rightarrow\mathbb{T}%
\]
is a continuous map from the compact metric space $[N]^{\ast}\cup\lbrack
N]^{\infty}$ into the space $(\mathbb{K}({\mathbb{R}}^{M}),d_{\tau})$.
\end{flushleft}

\begin{proof}
Using Lemma \ref{lemma02}, for $\theta=\theta_{1}\theta_{2}\cdots\theta_{l}%
\in\lbrack N]^{\ast}$ and $\theta^{-}=\theta_{1}\theta_{2}\cdots\theta_{l-1}%
$,
\begin{align*}
\mathbb{\pi}(\theta)  &  =E_{\theta}T_{e(\theta)}={\bigsqcup_{i=1}^{N}%
}E_{\theta}E_{e(\theta),i}T_{k-a_{i}}\\
&  \supset E_{\theta}E_{e(\theta),\theta_{l}}T_{k-a_{\theta_{l}}}%
=E_{\theta^{-}}T_{e(\theta^{-})}=\mathbb{\pi}(\theta^{-})\text{.}%
\end{align*}
It follows that $\{\pi(\theta|k)\}$ is an increasing sequence of tilings for
all $\theta\in\lbrack N]^{\infty}$, as in Equation (\ref{eqthmONE}), and so
converges to a well-defined limit. Since the maps in the IFS are strict
contractions, their inverses are expansive, whence $\pi(\theta)$ is unbounded
for all $\theta\in\lbrack N]^{\infty}$.

The fact that the tiles here are indeed tiles as we defined them at the start
of this paper follows from three readily checked observations. (i) The tiles
are nonempty perfect compact sets because they are isometric to the attractor,
that is not a singleton, of an IFS of similitudes. (ii) There are only
finitely many tiles that intersect any ball of finite radius. (iii) Any two
tiles can meet only on a set that is contained in the image under a similitude
of the set of overlap.

Next we prove that there are exactly $a_{\max}$ distinct tiles, up to
isometry, in any tiling $\pi(\theta)$ for $\theta\in\lbrack N]^{\infty}$. The
tiles of $\pi(\theta)$ take the form $\{f_{_{-\theta|k}}f_{\sigma}%
(A):\sigma\in\Omega_{e(\theta|k)}\}$ for some $k\in\mathbb{N}$. The mappings
here are similitudes whose scaling factors are $\{s^{e(\sigma)-e(\theta
|k)}:e(\sigma)-e(\theta|k)>0\geq e(\sigma)-e(\theta|k)-a_{|\sigma|}\},$ namely
$\{s^{m}:m>0\geq m-a_{|\sigma|}\}$ for which the possible values are at most
all of $\{1,2,\dots,a_{\max}\}$. That all of these values occur for large
enough $k$ follows from $\gcd\{a_{i}:i=1,2,\dots, N\}=1$.

Next we prove that $\pi:[N]^{\ast}\cup\lbrack N]^{\infty}\rightarrow
\mathbb{T}$ is a continuous map from the compact metric space $[N]^{\ast}%
\cup\lbrack N]^{\infty}$ onto the space $(\mathbb{T},d_{T}).$ The map
$\pi|_{[N]^{\ast}}:[N]^{\ast}\rightarrow\mathbb{T}$ is continuous on the
discrete part of the space $([N]^{\ast},d_{[N]^{\ast}\cup\lbrack N]^{\infty}%
})$ because each point $\theta\in\lbrack N]^{\ast}$ possesses an open
neighborhood that contains no other points of $[N]^{\ast}\cup\lbrack
N]^{\infty}$. To show that $\pi$ is continuous at points of $[N]^{\infty}$ we
follow a similar method to the one in \cite{anderson}. Let $\varepsilon>0$ be
given and let $B(R)$ be the open ball of radius $R$ centered at the origin.
Choose $R$ so large that $h(\rho(\overline{B(R)}),\mathbb{S}^{M})<\varepsilon
$. This implies that if two tilings differ only where they intersect the
complement of $\overline{B(R)}$, then their distance $d_{\tau}$ apart is less
than $\varepsilon$. But geometrical consideration of the way in which
\textit{support(}$\pi(\theta_{1}\theta_{2}\theta_{3}..\theta_{k}))$ grows with
increasing $k$ shows that we can choose $K$ so large that \textit{support(}%
$\pi(\theta_{1}\theta_{2}\theta_{3}..\theta_{k}))\cap\overline{B(R)}$ is
constant for all $k\geq K$. It follows that%
\[
h(\rho(\pi(\theta_{1}\theta_{2}..\theta_{k})),\rho(\pi(\theta_{1}\theta
_{2}..\theta_{l})))\leq\varepsilon
\]
and as a consequence
\[
h(\rho(\partial\pi(\theta_{1}\theta_{2}..\theta_{k})),\rho(\partial\pi
(\theta_{1}\theta_{2}..\theta_{l})))\leq\varepsilon
\]
for all $k,l\geq K$. It follows that $h(\rho(\pi(\theta)),\rho(\pi
(\omega)))\leq\varepsilon)$ whenever $\theta_{1}\theta_{2}..\theta_{K}%
=\omega_{1}\omega_{2}..\omega_{K}$. It follows that $\pi$ is continuous.
\end{proof}

\section{\label{proofofTHREE}Theorem \ref{theoremTHREE}: When Do all Tilings
Repeat the Same Patterns?}

\begin{flushleft}
\textbf{Theorem 2.}
\end{flushleft}

\begin{enumerate}
\item Each unbounded tiling in $\mathbb{T}$ is quasiperiodic and all tilings
in $\mathbb{T}$ have the local isomorphism property.

\item If $\theta$ is eventually periodic, then $\pi(\theta)$ is self-similar.
In fact, if $\theta=\alpha\overline{\beta}$ for some $\alpha,\beta\in\left[
N\right]  ^{\ast},$ then $f_{-\alpha}f_{-\beta}\left(  f_{-\alpha}\right)
^{-1}$ is a self-similarity of $\pi(\theta)$.
\end{enumerate}

\begin{proof}
(1) First we prove quasiperiodicity. This is related to the self-similarity of
the sequence of tilings $\left\{  T_{k}\right\}  $ mentioned in Proposition
\ref{lemmass}.

Let $\theta\in\left[  N\right]  ^{\infty}$ be given and let $P$ be a patch in
$\pi(\theta)$. There is a $K_{1}\in\mathbb{N}$ such that $P$ is contained in
$\pi(\theta|K_{1})$. Hence an isometric copy of $P$ is contained in $T_{K_{2}%
}$ where $K_{2}=e(\theta|K_{1})$. Now choose $K_{3}\in\mathbb{N}$ so that an
isometric copy of $T_{K_{2}}$ is contained in each $T_{k}$ with $k\geq K_{3}.$
That this is possible follows from the recursion (\ref{Tkformula2}) of Lemma
\ref{lemma02} and \textit{gcd}$\{a_{i}\}=1$. In particular, $T_{K_{2}}\subset
T_{K_{3}+i}$ for all $i\in\{1,2,...,a_{\max}\}$.

Now let $K_{4}=K_{3}+a_{\max}$. Then, for all $k\geq K_{4}$, the tiling
$T_{k}$ is an isometric combination of $\{T_{K_{3}+i}:$\textit{ }%
$i=1,2,...,a_{\max}\}$, and each of these tilings contains a copy of
$T_{K_{2}}$ and in particular a copy of $P$.

Let $D=\max\{\left\Vert x-y\right\Vert :x,y\in A\}$ be the diameter of $A$.
The support of $T_{k}$ is $s^{-k}A$ which has diameter $s^{-k}D.$ Hence
\textit{support}$(T_{k})\subset$ $B(x,2s^{-k}D)$, the ball centered at $x$ of
radius $2s^{-k}D$, for all $x\in$ \textit{support}$(T_{k})$. It follows that
if $x\in$\textit{support}$\pi(\theta^{\prime})$ for any $\theta^{\prime}%
\in\left[  N\right]  ^{\infty}$, then $B(x,2s^{-K_{4}}D)$ contains a copy of
\textit{support}$(T_{K_{2}})$ and hence a copy of $P$. Therefore all unbounded
tilings in $\mathbb{T}$ are quasiperiodic.

In \cite{Rad} Radin and Wolff define a tiling to have the local ismorphism
property if for every patch $P$ in the tiling there is some distance $d(P)$
such that every sphere of diameter $d(P)$ in the tiling contains an isometric
copy of $P$. Above, we have proved a stronger property of tilings, as defined
here, of fractal blow-ups. Given $P,$ there is a distance $d(P)$ such that
each sphere of diameter $d(P),$ centered at any point belonging to the support
of any unbounded tiling in $\mathbb{T}$, contains a copy of $P$.

(2) Let $\theta=\alpha\overline{\beta}=\alpha_{1}\alpha_{2}\cdots\alpha
_{l}\beta_{1}\beta_{2}\cdots\beta_{m}\beta_{1}\beta_{2}\cdots\beta_{m}%
\beta_{1}\beta_{2}\cdots\beta_{m}\cdots$. We have the equivalent increasing
unions
\[
\pi(\theta)=%
{\textstyle\bigcup\limits_{k\in\mathbb{N}}}
E_{\theta|k}T_{e(\theta|k)}=%
{\textstyle\bigcup\limits_{j\in\mathbb{N}}}
E_{\theta|(l+jm)}T_{e(\theta|(l+jm))}=%
{\textstyle\bigcup\limits_{j\in\mathbb{N}}}
E_{\theta|(l+jm+m)}T_{e(\theta|(l+jm+m))}%
\]
where, for all $k$,
\[
E_{\theta|k}=f_{-\theta|k}s^{e(\theta|k)}\text{.}%
\]
We can write
\[
\pi(\theta)=%
{\textstyle\bigcup\limits_{j\in\mathbb{N}}}
E_{\theta|(l+jm)}T_{e(\theta|(l+jm))}=f_{-\alpha}%
{\textstyle\bigcup\limits_{j\in\mathbb{N}}}
f_{-\beta}^{j}s^{e(\theta|(l+jm))}T_{e(\theta|(l+jm))},
\]
and also
\[
\pi(\theta)=%
{\textstyle\bigcup\limits_{j\in\mathbb{N}}}
E_{\theta|(l+jm+m)}T_{e(\theta|(l+jm+m))}=f_{-\alpha}f_{-\beta}%
{\textstyle\bigcup\limits_{j\in\mathbb{N}}}
f_{-\beta}^{j}s^{e(\theta|(l+jm+m))}T_{e(\theta|(l+jm+m))}\text{.}%
\]

Here $f_{-\beta}^{j}s^{e(\theta|(l+jm+m))}T_{e(\theta|(l+jm+m))}$ is a
refinement of $f_{-\beta}^{j}s^{e(\theta|(l+jm))}T_{e(\theta|(l+jm))}$. It
follows that $\left(  f_{-\alpha}f_{-\beta}\right)  ^{-1}\pi(\theta)$ is a
refinement of $\left(  f_{-\alpha}\right)  ^{-1}\pi(\theta)$, from which it
follows that $\left(  f_{-\alpha}\right)  \left(  f_{-\alpha}f_{-\beta
}\right)  ^{-1}\pi(\theta)$ is a refinement of $\pi(\theta)$. Therefore, every
set in $\left(  f_{-\alpha}f_{-\beta}\right)  \left(  f_{-\alpha}\right)
^{-1}\pi(\theta)$ is a union of tiles in $\pi(\theta)$.
\end{proof}

\section{\label{realabsec} Relative and Absolute Addresses}

In order to understand how different tilings relate to one another, the
notions of relative and absolute addresses of tiles are introduced. Given an
IFS $\mathcal{F}$, the \textit{set of absolute addresses} is defined to be:
\[
\mathbb{A}:=\{\theta.\omega:\theta\in\lbrack N]^{\ast},\,\omega\in
\Omega_{e(\theta)},\,\theta_{\left\vert \theta\right\vert }\neq\omega_{1}\}.
\]
Define $\widehat{\pi}:\mathbb{A\rightarrow\{}t\in T:T\in\mathbb{T\}}$ by
\[
\widehat{\pi}(\theta.\omega)=f_{-\theta}.f_{\omega}(A).
\]
We say that $\theta.\omega$ is an \textit{absolute address} of the tile
$f_{-\theta}.f_{\omega}(A)$. It follows from Definition \ref{defONE} that the
map $\widehat{\pi}$ is surjective: every tile of $\mathbb{\{}t\in
T:T\in\mathbb{T\}}$ possesses at least one address. The condition
$\theta_{\left\vert \theta\right\vert }\neq\omega_{1}$ is imposed to make
cancellation unnecessary.

The \textit{set of relative addresses} is associated with the tiling $T_{k}$
of $A_{k}=s^{-k}A$ and is defined to be $\{.\omega:\omega\in\Omega_{k}\}$.

\begin{proposition}
\label{lembij}There is a bijection between the set of relative addresses
$\{.\omega:\omega\in\Omega_{k}\}$ and the tiles of $T_{k}$, for all
$k\in\mathbb{N}_{0}$.
\end{proposition}

\begin{proof}
This follows from the non-overlapping union
\[
A=%
{\textstyle\bigsqcup\limits_{\omega\in\Omega_{k}}}
f_{\omega}(A)\text{.}%
\]
This expression follows immediately from Lemma \ref{lemstruc3}; see the start
of the proof of Lemma \ref{lemma02}.
\end{proof}

Accordingly, we say that $.\omega$, or equivalently $\varnothing.\omega,$
where $\omega\in\Omega_{k}$, is \textit{the relative address} of the tile
$s^{-k}f_{\omega}(A)$ in the tiling $T_{k}$ of $A_{k}$. Note that a tile of
$T_{k}$ may share the same relative address as a different tile of $T_{l}$ for
$l\neq k$.

Define the \textit{set of labelled tiles} of $T_{k}$ to be%
\[
\mathcal{A}_{k}=\{(.\omega,s^{-k}f_{\omega}(A)):\omega\in\Omega_{k}\}
\]
for all $k\in\mathbb{N}_{0}$. A key point about relative addresses is that the
set of labelled tiles of $T_{k}$ for $k\in\mathbb{N}$ can be computed
recursively. Define
\[
\mathcal{A}_{k}^{^{\prime}}=\{(\omega,s^{-k}f_{\omega}(A))\in\mathcal{A}%
_{k}:e(\omega)=k+1\}\subset\mathcal{A}_{k}\text{.}%
\]
An example of the following inductive construction is illustrated in\ Figure
\ref{l-maps}, and some corresponding tilings $\pi(\theta)$ labelled by
absolute addresses are illustrated in Figure \ref{absolute}.

\begin{lemma}
\label{branchlem}For all $k\in\mathbb{N}_{0}$ we have%
\[
\mathcal{A}_{k+1}=\mathcal{L(A}_{k}\backslash\mathcal{A}_{k}^{^{\prime}}%
)\cup\mathcal{M(A}_{k}^{^{\prime}})
\]
where
\begin{align*}
\mathcal{L}(\omega,s^{-k}f_{\omega}(A))  &  =(\omega,s^{-k-1}f_{\omega}(A)),\\
\mathcal{M}(\omega,s^{-k}f_{\omega}(A))  &  =\big \{(\omega i,s^{-k-1}%
f_{\omega i}(A)):i\in\lbrack N]\big \}\text{.}%
\end{align*}

\end{lemma}

\begin{proof}
This follows immediately from Lemma \ref{lemstruc2}.
\end{proof}

\section{\label{strongsec}Strong Rigidity, Definition of ``Amalgamation and
Shrinking" Operation $\alpha$ on Tilings, and Proof of Theorem
\ref{intersectthm}.}

We begin this key section by introducing an operation, called
\textquotedblleft amalgamation and shrinking", that maps certain tilings into
tilings. This leads to the main result of this section, Theorem
\ref{intersectthm}, which, in turn, leads to Theorem \ref{theoremTWO}.

\begin{definition}
\label{rigiddef} Let $T_{0}=\{f_{i}(A):i\in\left[  N\right]  \}$. The IFS
$\mathcal{F}$ is said to be \textbf{rigid} if (i) there exists no non-identity
isometry $E\in\mathcal{E}$ such that $T_{0}\cap ET_{0}$ is non-empty and tiles
$A\cap ET$, and (ii) there exists no non-identity isometry $E\in\mathcal{E}$
such that $A=EA$.
\end{definition}

\begin{definition}
Define $\mathbb{T}^{\prime}$ to be the set of all tilings using the set of
prototiles $\left\{  s^{i}A:i=1,2,...,a_{\max}\right\}  $. Any tile that is
isometric to $s^{a_{\max}}A$ is called a \textbf{small tile}, and any tile
that is isometric to $sA$ is called a \textbf{large tile}. We say that a
tiling $P\in\mathbb{T}^{\prime}$ comprises a set of \textbf{partners }if
$P=ET_{0}$ for some $E\in\mathcal{E}$. Define $\mathbb{T^{\prime\prime}\subset
T}^{\prime}$ to be the set of all tilings in $\mathbb{T}^{\prime}$ such that,
given any $Q\in\mathbb{T^{\prime\prime}}$ and any small tile $t\in Q$, there
is a set of partners of $t$, call it $P(t)$, such that $P(t)\subset Q$. Given
any $Q\in\mathbb{T^{\prime\prime}}$ we define $Q^{\prime}$ to be the union of
all sets of partners in $Q$.
\end{definition}

\begin{definition}
Let $\mathcal{F}$ be a rigid IFS. The amalgamation and shrinking operation
$\alpha:\mathbb{T^{\prime\prime}\rightarrow T}^{\prime}$ is defined by
\[
\alpha Q=\{st:t\in Q\backslash Q^{\prime}\}\cup%
{\displaystyle\bigsqcup_{\{E\in\mathcal{E}:ET_{0}\subset Q^{\prime}\}}}
sEA\text{.}%
\]

\end{definition}

\begin{lemma}
\label{inverselemma} If $\mathcal{F}$ is rigid, the function $\alpha
:\mathbb{T^{\prime\prime}\rightarrow T}^{\prime}$is well-defined and
bijective; in particular, $\alpha^{-1}:\mathbb{T}^{\prime}\rightarrow
\mathbb{T^{\prime\prime}}$ is well defined by
\[
\alpha^{-1}(Q)=\{\alpha_{Q}^{-1}(q):q\in Q\}
\]
where
\[
\alpha_{Q}^{-1}(q)=\left\{
\begin{array}
[c]{c}%
s^{-1}q\text{ if }q\in Q\text{ is not a large tile}\\
s^{-1}ET_{0}\text{ if }Eq\text{ is a large tile, some }E\in\mathcal{E}%
\end{array}
\right.
\]

\end{lemma}

\begin{proof}
Because $\mathcal{F}$ is rigid, there can be no ambiguity with regard to which
sets of tiles in a tiling are partners, nor with regard to which tiles are the
partners of a given small tile. Hence $\alpha:\mathbb{T^{\prime\prime
}\rightarrow T}^{\prime}$ is well defined. Given any $T^{\prime}\in
\mathbb{T}^{\prime}$ we can find a unique $Q\in\mathbb{T^{\prime\prime}}$ such
that $\alpha(Q)=T^{\prime},$ namely $\alpha^{-1}(Q)$ as defined in the lemma.
\end{proof}

\begin{lemma}
\label{alphaTlem}Let $\mathcal{F}$ be rigid and $k\in\mathbb{N}$. Then

(i) $T_{k}\in\mathbb{T^{\prime\prime}}$;

(ii) $\alpha T_{k}=T_{k-1}$ and $\alpha^{-1}T_{k-1}=T_{k}$.
\end{lemma}

\begin{proof}
As described in Lemma \ref{branchlem}$,$ $T_{k}$ can constructed in a
well-defined manner, starting from from $T_{k-1}$, by scaling and splitting,
that is, by applying $\alpha^{-1}$. Conversely $T_{k-1}$ can be constructed
from $T_{k}$ by applying $\alpha$. Statements (i) and (ii) are consequences.
\end{proof}

\begin{lemma}
\label{srinterlem} If $\mathcal{F}$ is rigid, $L,M\in\mathbb{T^{\prime\prime}%
}$, and $L\cap M$ tiles support($L)\,\cap\,$support($M),$ then $L\, \cap\,
M\in\mathbb{T^{\prime\prime}}$. Moreover,
\[
\alpha(L\cap M)=\alpha(L)\cap\alpha(M),
\]
and $\alpha(L\cap M)$ tiles support$\, \alpha(L)\, \cap\, $support$\,
\alpha(M)$.
\end{lemma}

\begin{proof}
Since $L,M\in\mathbb{T^{\prime\prime}\subset T}^{\prime}$ lie in the range of
$\alpha^{-1},$ we can find unique $L^{\prime},M^{\prime}\in\mathbb{T}^{\prime
}$ such that
\[
L=\alpha^{-1}L^{\prime}\text{ and }M=\alpha^{-1}M^{\prime}.
\]
Note that $\alpha^{-1}(T^{\prime})=\left\{  \alpha^{-1}(t):t\in T^{\prime
}\right\}  $ for all $T^{\prime}\in\mathbb{T}^{\prime}$, which implies that
$\alpha^{-1}$ commutes both with unions of disjoint tilings and also with
intersections of tilings whose intersections tile the intersections of their
supports. It follows that $L\cap M\in\mathbb{T^{\prime\prime}}$,
\begin{align*}
\alpha(L\cap M)  &  =\alpha(\alpha^{-1}L^{\prime}\cap\alpha^{-1}M^{\prime})\\
&  =\alpha(\alpha^{-1}(L^{\prime}\cap M^{\prime}))\\
&  =L^{\prime}\cap M^{\prime}\\
&  =\alpha\left(  L\right)  \cap\alpha\left(  M\right)  \text{,}%
\end{align*}
and support $\alpha(L\cap M)=$\, support $\, \alpha\left(  L\right)  \cap
\,$support $\alpha\left(  M\right)  $.
\end{proof}

\begin{definition}
\label{strongdef}$\mathcal{F}$ is \textbf{strongly rigid} if $\mathcal{F}$ is
rigid and whenever $i,j\in\{0,1,2,\dots,a_{\max}-1\},E\in\mathcal{E}$, and
$T_{i}\cap ET_{j}$ tiles $A_{i}\cap EA_{j}$, either $T_{i}\subset ET_{j}$ or
$T_{i}\supset ET_{j}.$
\end{definition}

Section \ref{exsec} contain a few examples of strongly rigid IFSs.

\begin{lemma}
\label{intersectlemma} Let $\mathcal{F}$ be strongly rigid, $k,l\in
\mathbb{N}_{0}$, and $E\in\mathcal{E}$.

(i) If $ET_{k}\cap T_{k}$ is nonempty and tiles $EA_{k}\cap A_{k},$ then
$E=id$.

(ii) If $EA_{k}\cap A_{k+l}$ is nonempty and $ET_{k}\cap T_{k+l}$ tiles
$EA_{k}\cap A_{k+l}$, then $ET_{k}\subset T_{k+l}$.
\end{lemma}

\begin{proof}
Suppose $ET_{k}\cap T_{l}\neq\varnothing$ and t.i.s. (tiles intersection of
supports). Without loss of generality assume $k\leq l$, for if not, then apply
$E^{-1}$, then redefine $E^{-1}$ as $E$.

Both $ET_{k}$ and $T_{l}$ lie in the domain of $\alpha^{k}$, so we can apply
Lemma \ref{srinterlem} $k$ times, yielding
\begin{align}
\alpha^{k}(ET_{k}\cap T_{l})  &  =s^{k}Es^{-k}T_{0}\cap T_{l-k} \label{aboveq}%
\\
&  :=\widetilde{E}T_{0}\cap T_{l-k}\neq\varnothing,\nonumber
\end{align}
where $\widetilde{E}T_{0}\cap T_{l-k}$ t.i.s. Now observe that by Lemma
\ref{lemma02} we can write, for all $k^{\prime}\geq l^{\prime}+a_{\max}$,%
\[
T_{k^{\prime}}={\bigsqcup_{\omega\in\Omega_{l^{\prime}}}}E_{k^{\prime},\omega
}T_{k^{\prime}-e(\omega)}\left(  =\left\{  E_{k^{\prime},\omega}T_{k^{\prime
}-e(\omega)}:\omega\in\Omega_{l^{\prime}}\right\}  \right)  ,
\]
where $E_{k^{\prime},\omega}\in\mathcal{E}$ for all $k^{\prime},\omega$.
Choosing $l^{\prime}=k^{\prime}-a_{\max}$ and noting that, for $\omega
\in\Omega_{l^{\prime}}$, we have $e(\omega)\in\{l^{\prime}+1,\dots,l^{\prime
}+a_{\max}\},$ and for $\omega\in\Omega_{k^{\prime}-a_{\max}}$ we have
$e(\omega)\in\{k^{\prime}-a_{\max}+1,\dots,k^{\prime}\}.$ Therefore
$k^{\prime}-e(\omega)\in\{0,1,\dots,a_{\max}-1\}$ and we obtain the explicit
representation%
\[
T_{k^{\prime}}={\bigsqcup_{\omega\in\Omega_{k^{\prime}-a_{\max}}}}%
E_{k^{\prime},\omega}T_{k^{\prime}-e(\omega)}%
\]
which is an isometric combination of $\{T_{0},T_{1},\dots,T_{a_{\max}-1}\}$.
In particular, we can always reexpress $T_{l-k}$ in (\ref{aboveq}) as
isometric combination of $\{T_{0},T_{1},\dots,T_{a_{\max}-1}\}$ and so there
is some $E^{\prime}$ and some $T_{m}\in\{T_{0},T_{1},\dots,T_{a_{\max}-1}\}$
such that
\[
\widetilde{E}T_{0}\cap E^{\prime}T_{m}\neq\varnothing\text{ and t.i.s.}%
\]

By the strong rigidity assumption, this implies $\widetilde{E}T_{0}\subset
E^{\prime}T_{m}$, which in turn implies
\[
\widetilde{E}T_{0}\subset T_{l-k}%
\]
and t.i.s. Now apply $\alpha^{-k}$ to both sides of this last equation to
obtain the conclusions of the lemma.
\end{proof}

\begin{flushleft}
\textbf{Theorem \ref{intersectthm}.} Let $\mathcal{F}$ be strongly rigid. If
$\theta,\theta^{\prime}\in\lbrack N]^{\ast}$and $E\in\mathcal{E}$ are such
that $\pi(\theta)\cap E\pi(\theta^{\prime})$ is not empty and tiles
$A_{-\theta}\cap EA_{-\theta^{\prime}}$, then either $\pi(\theta)\subset
E\pi(\theta^{\prime})$ or $E\pi(\theta^{\prime})\subset\pi(\theta)$. In this
situation, if $e(\theta)=e(\theta^{\prime}),$ then $E\pi(\theta^{\prime}%
)=\pi(\theta).$
\end{flushleft}

\begin{proof}
This follows from Lemma \ref{intersectlemma}. If $\theta,\theta^{\prime}%
\in\lbrack N]^{\ast}$and $E\in\mathcal{E}$ are such that $\pi(\theta)\cap
E\pi(\theta^{\prime})$ is not empty and tiles $A_{-\theta}\cap EA_{-\theta
^{\prime}}$, then $\theta,\theta^{\prime}\in\lbrack N]^{\ast}$and
$E\in\mathcal{E}$ are such that $E_{\theta}T_{e(\theta)}\cap EE_{\theta
^{\prime}}T_{e(\theta^{\prime})}$ is not empty and tiles $E_{\theta
}A_{e(\theta)}\cap EE_{\theta^{\prime}}A_{e(\theta^{\prime})}$, where
$E_{\theta}=f_{-\theta}s^{e(\theta)}$ and $E_{\theta^{\prime}}=f_{-\theta
^{\prime}}s^{e(\theta^{\prime})}$ are isometries$.$ Assume, without loss of
generality, that $e(\theta)\leq e(\theta^{\prime})$ and apply $E_{\theta
^{\prime}}^{-1} E^{-1}$ to obtain that $\theta,\theta^{\prime}\in\lbrack
N]^{\ast}$ and $E^{\prime}=E_{\theta^{\prime}}^{-1}E^{-1} E_{\theta}%
\in\mathcal{E}$ are such that $E^{\prime}T_{e(\theta)}\cap T_{e(\theta
^{\prime})}$ is not empty and tiles $E^{\prime}A_{e(\theta)}\cap
A_{e(\theta^{\prime})}.$ By Lemma \ref{intersectlemma} it follows that
$E^{\prime}T_{e(\theta)}\subset T_{e(\theta^{\prime})}$, i.e. $E_{\theta
^{\prime}}^{-1}E^{-1}E_{\theta}T_{e(\theta)}\subset T_{e(\theta^{\prime})},$
i.e. $\pi(\theta)\subset E\pi(\theta^{\prime}).$ If also $e(\theta^{\prime
})\leq e(\theta)$ (i.e. $e(\theta^{\prime})=e(\theta)$), then also
$E\pi(\theta^{\prime})\subset\pi(\theta)$. Therefore $E\pi(\theta^{\prime
})=\pi(\theta).$
\end{proof}

\section{\label{proofofTWO}Theorem \ref{theoremTWO}: When is a Tiling
Non-Periodic?}

\begin{flushleft}
\textbf{Theorem \ref{theoremTWO}.} \textit{If }$F$\textit{ is strongly rigid,
then there does not exist any non-identity isometry }$E\in\mathcal{E}$\textit{
and }$\theta\in\lbrack N]^{\infty}$\textit{ such that }$E\pi(\theta)\subset
\pi(\theta)$\textit{.}
\end{flushleft}

\begin{proof}
Suppose there exists an isometry $E$ such that $E\pi(\theta)=\pi(\theta).$
Then we can choose $K\in\mathbb{N}_{0}$ so large that $E\pi(\theta|K)\cap
\pi(\theta|K)\neq\varnothing$ and $E\pi(\theta|K)\cap\pi(\theta|K)$ tiles
$EA_{-\theta|K}\cap A_{-\theta|K}.$ By Theorem \ref{intersectthm} it follows
that%
\[
E\pi(\theta|K)=\pi(\theta|K)
\]
This implies
\[
EE_{\theta}T_{e\left(  \theta|K\right)  }=E_{\theta}T_{e\left(  \theta
|K\right)  }%
\]
whence, because $E_{\theta}T_{e\left(  \theta|K\right)  }$ is in the domain of
$\alpha^{e\left(  \theta|K\right)  }$ and $\alpha^{e\left(  \theta|K\right)
}T_{e\left(  \theta|K\right)  }=T_{0},$ we have by Lemma \ref{alphaTlem}
\begin{align*}
\alpha^{e\left(  \theta|K\right)  }E  &  E_{\theta}T_{e\left(  \theta
|K\right)  } =\alpha^{e\left(  \theta|K\right)  }E_{\theta}T_{e\left(
\theta|K\right)  }\\
&  \Longrightarrow s^{e\left(  \theta|K\right)  }EE_{\theta}s^{-e\left(
\theta|K\right)  }\alpha^{e\left(  \theta|K\right)  }T_{e\left(
\theta|K\right)  }=s^{e\left(  \theta|K\right)  }E_{\theta}s^{-e\left(
\theta|K\right)  }\alpha^{e\left(  \theta|K\right)  }T_{e\left(
\theta|K\right)  }\\
&  \Longrightarrow s^{e\left(  \theta|K\right)  }EE_{\theta}s^{-e\left(
\theta|K\right)  }T_{0}=s^{e\left(  \theta|K\right)  }E_{\theta}s^{-e\left(
\theta|K\right)  }T_{0}\\
&  \Longrightarrow s^{e\left(  \theta|K\right)  }EE_{\theta}s^{-e\left(
\theta|K\right)  }=s^{e\left(  \theta|K\right)  }E_{\theta}s^{-e\left(
\theta|K\right)  }\text{ (using rigidity)}\\
&  \Longrightarrow E=id\text{.}%
\end{align*}

\end{proof}

It follows that if $\mathcal{F}$ is strongly rigid, then $\pi(\theta)$ is
non-periodic for all $\theta$.

\section{\label{invertsec}When is $\pi:[N]^{\ast}\cup\lbrack N]^{\infty
}\rightarrow\mathbb{T}$ invertible?}

\begin{lemma}
\label{noninjlem}For all $\mathcal{F}$ the restricted mapping $\pi|_{\left[
N\right]  ^{\ast}\text{.}}:[N]^{\ast}\rightarrow\mathbb{T}$ is injective.
\end{lemma}

\begin{proof}
To simplify notation, write $\pi=\pi|_{\left[  N\right]  ^{\ast}\text{.}}$ We
show how to calculate $\theta$ given $\pi\left(  \theta\right)  $ for
$\theta\in\left[  N\right]  ^{\ast}.$ By Lemma~\ref{lemma02} we have
$\pi(\theta)=E_{\theta}T_{e(\theta)}$, where $E$ is the isometry $f_{-\theta
}s^{e(\theta)}$. Given $\pi(\theta)$, we can calculate
\[
e(\theta)=\frac{\ln\left\vert A\right\vert -\ln\left\vert \pi(\theta
)\right\vert }{\ln s},
\]
where $\left\vert U\right\vert $ denotes the diameter of the set $U$.

We next show that if $E_{\theta}=E_{\theta^{\prime}}$ for some $\theta
\neq\theta^{\prime}$ with $e(\theta)=e(\theta^{\prime})$, then $\pi
(\theta)\neq\pi(\theta^{\prime})$. To do this, suppose that $E_{\theta
}=E_{\theta^{\prime}}$. This implies that $f_{-\theta}=f_{-\theta^{\prime}}$
which implies
\[
\left(  f_{-\theta^{\prime}}\right)  ^{-1}f_{-\theta}=id\text{,}%
\]
which is not possible when $\theta\neq\theta^{\prime},$ as we prove next. The
similitude $\left(  f_{-\theta^{\prime}}\right)  ^{-1}f_{-\theta}$ maps
$\left(  f_{-\theta}\right)  ^{-1}(A)\subset A$ to $\left(  f_{-\theta
^{\prime}}\right)  ^{-1}(A)\subset A$, and these two subsets of\ $A$ are
distinct for all $\theta,\theta^{\prime}\in\left[  N\right]  ^{\ast}$with
$\theta\neq\theta^{\prime}$, as we prove next.

Let $\omega,\omega^{\prime}$ denote the two strings $\theta,\theta^{\prime}$
written in inverse order, so that $\theta\neq\theta^{\prime}$ is equivalent to
$\omega\neq\omega^{\prime}$. First suppose $\left\vert \omega\right\vert
=\left\vert \omega^{\prime}\right\vert =m$ for some $m\in\mathbb{N}.$ Then
use
\[
A=%
{\displaystyle\bigsqcup\limits_{\omega\in\left[  N\right]  ^{m}}}
f_{\omega}(A),
\]
which tells us that $f_{\omega}(A)$ and $f_{\omega^{\prime}}(A)$ are disjoint.
Since $\left(  f_{-\theta^{\prime}}\right)  ^{-1}f_{-\theta}$ maps
\linebreak$\left(  f_{-\theta}\right)  ^{-1}(A)=f_{\omega}(A)$ to the distinct
set $\left(  f_{-\theta^{\prime}}\right)  ^{-1}(A)=f_{\omega^{\prime}}(A)$, we
must have $\left(  f_{-\theta^{\prime}}\right)  ^{-1}f_{-\theta}\neq id$.

Now suppose $\left\vert \omega\right\vert =m<\left\vert \omega^{\prime
}\right\vert =m^{\prime}.$ If both strings $\omega$ and $\omega^{\prime}$
agree through the first $m$ places, then $f_{\omega}(A)$ is a strict subset of
$f_{\omega^{\prime}}^{-1}(A)$ and again we cannot have $\left(  f_{-\theta
^{\prime}}\right)  ^{-1}f_{-\theta}=id$. If both strings $\omega$ and
$\omega^{\prime}$ do not agree through the first $m$ places, then let $p<m$ be
the index of their first disagreement. Then we find that $f_{\omega}(A)\ $is a
subset of $f_{\omega|p}(A)$, while $f_{\omega^{\prime}}(A)$ is a subset of the
set $f_{\omega^{\prime}|p}(A)$, which is disjoint from $f_{\omega|p}(A)$.
Since $\left(  f_{-\theta^{\prime}}\right)  ^{-1}f_{-\theta}$ maps $f_{\omega
}(A)$ to $f_{\omega^{\prime}}(A)$, we again have that $\left(  f_{-\theta
^{\prime}}\right)  ^{-1}f_{-\theta}\neq id$.
\end{proof}

We are going to need a key property of certain shifts maps on tilings, defined
in the next lemma.

\begin{lemma}
The mappings $S_{i}:\{\pi(\theta):\theta\in\lbrack N]^{l}\cup\lbrack
N]^{\infty},l\geq a_{i}\}\rightarrow\mathbb{T}^{\prime}$ for $i\in\lbrack N]$
are well-defined by%
\[
S_{i}=f_{i}s^{-a_{i}}\alpha^{a_{i}}\text{.}%
\]
It is true that%
\[
S_{\theta_{1}}\pi(\theta)=\pi(S\theta)
\]
for all $\theta\in\lbrack N]^{l}\cup\lbrack N]^{\infty}$ where $l\geq
a_{\theta_{1}}$.
\end{lemma}

\begin{proof}
We only consider the case $\theta\in\lbrack N]^{\infty}$. The case $\theta
\in\lbrack N]^{l}$ is treated similarly. A detailed calculation, outlined
next, is needed. The key idea is that $\pi\left(  \theta\right)  $ is broken
up into a countable union of disjoint tilings, each of which belongs to the
domain of $\alpha^{k}$ for all $k\leq K$ for any $K\in\mathbb{N}$. For all
$K\in\mathbb{N}$ we have:%
\[
\pi\left(  \theta\right)  =E_{\theta|K}T_{e\left(  \theta|K\right)  }%
{\textstyle\bigsqcup}
{\textstyle\bigsqcup_{k=K}^{\infty}}
E_{\theta|k+1}T_{e\left(  \theta|k+1\right)  }\backslash E_{\theta
|k}T_{e\left(  \theta|k\right)  }\text{.}%
\]
The tilings on the r.h.s. are indeed disjoint, and each set belongs to the
domain of $\alpha^{e\left(  \theta|K\right)  }$, so we can use Lemma
\ref{srinterlem} applied countably many times to yield%
\[
S_{\theta_{1}}\pi\left(  \theta\right)  =S_{\theta_{1}}\left(  E_{\theta
|K}T_{e\left(  \theta|K\right)  }\right)
{\textstyle\bigsqcup_{k=K}^{\infty}}
S_{\theta_{1}}\left(  E_{\theta|k+1}T_{e\left(  \theta|k+1\right)  }\right)
\backslash S_{\theta_{1}}\left(  E_{\theta|k}T_{e\left(  \theta|k\right)
}\right)  \text{.}%
\]
Evaluating, we obtain successively
\begin{align*}
S_{\theta_{1}}\pi\left(  \theta\right)   &  =f_{\theta_{1}}s^{-a_{\theta_{1}}%
}\alpha^{a_{\theta_{1}}}\left(  E_{\theta|K}T_{e\left(  \theta|K\right)
}\right)
{\textstyle\bigsqcup_{k=K}^{\infty}}
f_{\theta_{1}}s^{-a_{\theta_{1}}}\alpha^{a_{\theta_{1}}}\left(  E_{\theta
|k+1}T_{e\left(  \theta|k+1\right)  }\right)  \backslash f_{\theta_{1}%
}s^{-a_{\theta_{1}}}\alpha^{a_{\theta_{1}}}\left(  E_{\theta|k}T_{e\left(
\theta|k\right)  }\right)  ,\\
S_{\theta_{1}}\pi\left(  \theta\right)   &  =f_{\theta_{1}}E_{\theta
|K}s^{-a_{\theta_{1}}}\alpha^{a_{\theta_{1}}}T_{e\left(  \theta|K\right)  }%
{\textstyle\bigsqcup_{k=K}^{\infty}}
f_{\theta_{1}}E_{\theta|k+1}s^{-a_{\theta_{1}}}\alpha^{a_{\theta_{1}}%
}T_{e\left(  \theta|k+1\right)  }\backslash f_{\theta_{1}}E_{\theta
|k+1}s^{-a_{\theta_{1}}}\alpha^{a_{\theta_{1}}}T_{e\left(  \theta|k\right)
},\\
S_{\theta_{1}}\pi\left(  \theta\right)   &  =f_{\theta_{1}}E_{\theta
|K}s^{-a_{\theta_{1}}}T_{e\left(  S\theta|K-1\right)  }%
{\textstyle\bigsqcup_{k=K}^{\infty}}
f_{\theta_{1}}E_{\theta|k+1}s^{-a_{\theta_{1}}}T_{e\left(  S\theta|k\right)
}\backslash f_{\theta_{1}}E_{\theta|k}s^{-a_{\theta_{1}}}T_{e\left(
S\theta|k-1\right)  },\\
S_{\theta_{1}}\pi\left(  \theta\right)   &  =E_{S\theta|\left(  K-1\right)
}T_{e\left(  S\theta|K-1\right)  }%
{\textstyle\bigsqcup_{k=K}^{\infty}}
E_{S\theta|k}T_{e\left(  S\theta|k-1\right)  }\backslash E_{S\theta
|k-1}T_{e\left(  S\theta|k-1\right)  }=\pi\left(  S\theta\right)  .
\end{align*}

\end{proof}

\begin{flushleft}
\textbf{Theorem \ref{1to1thm}.} If $\pi(i)\cap\pi(j)$ does not tile $\left(
support\text{ }\pi(i)\right)  \cap\left(  support\text{ }\pi(j)\right)  $ for
all $i\neq j$, then $\pi:[N]^{\ast}\cup\lbrack N]^{\infty}\rightarrow
\mathbb{T}$ is one-to-one.
\end{flushleft}

\begin{proof}
The map $\pi$ is one-to-one on $[N]^{\ast}$ by Lemma \ref{noninjlem}, so we
restrict attention to points in $[N]^{\infty}$. If $\theta$ and $\theta
^{\prime}$ are such that $\theta_{1}=i$ and $\theta_{1}^{\prime}=j$, then the
result is immediate because $\pi(\theta)$ contains $\pi(i)$ and $\pi
(\theta^{\prime})$ contains $\pi(j)$. If $\theta$ and $\theta^{\prime}$ agree
through their first $K$ terms with $K\geq1$ and $\theta_{K+1}\neq$
$\theta_{K+1}^{\prime}$, then $\pi(S^{K}\theta)\neq\pi(S^{K}\theta^{\prime})$.
Now apply $S_{\theta_{1}}^{-1}S_{\theta_{2}}^{-1}...S_{\theta_{K}}^{-1}$ to
obtain $\pi(\theta)\neq\pi(\theta^{\prime})$. (We can do this last step
because $S_{i}^{-1}=\left(  f_{i}s^{-a_{i}}\alpha^{a_{i}}\right)  ^{-1}%
=\alpha^{-a_{i}}s^{a_{i}}f_{i}^{-1}$ has as its domain all of $\mathbb{T}%
^{\prime}$ and maps $\mathbb{T}^{\prime}$ into $\mathbb{T}^{\prime}$.)
\end{proof}

\section{\label{exsec}Examples}

\subsection{Golden b tilings}

A \textit{golden b} $G\subset{\mathbb{R}}^{2}$ is illustrated in Figure
\ref{golden-01}. This hexagon is the only rectilinear polygon that can be
tiled by a pair of differently scaled copies of itself \cite{S, Sch}.
Throughout this subsection the IFS is
\[
\mathcal{F}=\{\mathbb{R}^{2};f_{1},f_{2}\}
\]
where
\[
f_{1}(x,y)=%
\begin{pmatrix}
0 & s\\
-s & 0
\end{pmatrix}%
\begin{pmatrix}
x\\
y
\end{pmatrix}
+%
\begin{pmatrix}
0\\
s
\end{pmatrix}
,\quad f_{2}(x,y)=%
\begin{pmatrix}
-s^{2} & 0\\
0 & s^{2}%
\end{pmatrix}%
\begin{pmatrix}
x\\
y
\end{pmatrix}
+%
\begin{pmatrix}
1\\
0
\end{pmatrix}
,
\]
where the scaling ratios $s$ and $s^{2}$ obey $s^{4}+s^{2}=1$, which tells us
that $s^{-2}=\alpha^{-2}$ is the golden mean. The attractor of $\mathcal{F}$
is $A=G$. It is the union of two prototiles $f_{1}(G)$ and $f_{2}(G)$. Copies
of these prototiles are labelled $L$ and $S$. In this example, note that
$e(\theta)=\theta_{1}+\theta_{2}+\cdots+\theta_{\left\vert \theta\right\vert
}$ for $\theta\in\lbrack2]^{\ast}$.

\begin{figure}[h]
\centering
\includegraphics[width=5cm, keepaspectratio]{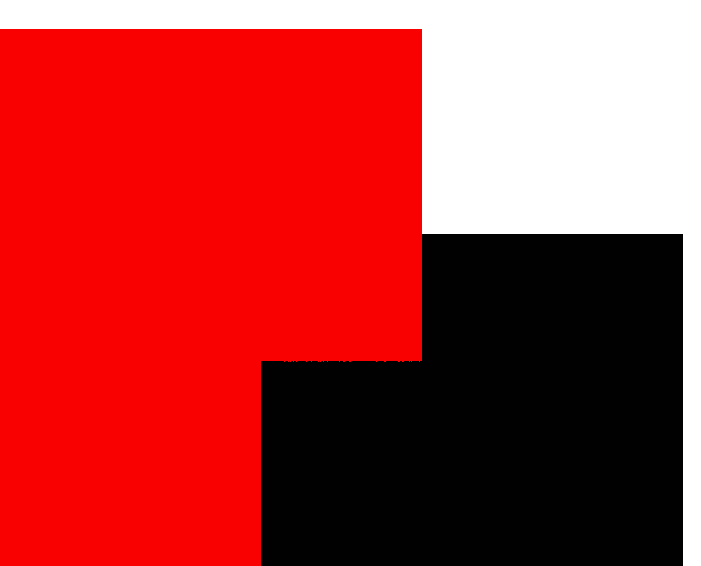}\caption{A golden b
is a union of two tiles, a small one and its partner, a large one. The
vertices of this golden b are located at $(0,0)$ $(1,0)$ $(1,\alpha^{3})$
$(\alpha^{2},\alpha^{3})$ $(\alpha^{2},\alpha)$ $(0,\alpha)$ in
counterclockwise order, starting at the lower left corner, where $\alpha^{-2}$
is the golden mean. This picture also represents a tiling $T_{0}%
=\pi(\varnothing)$. }%
\label{golden-01}%
\end{figure}

The figures in this section illustrate some earlier concepts in the context of
the golden b. Using some of these figures, it is easy to check that
$\mathcal{F}$ is strongly rigid, so the tilings $\pi(\theta)$ have all of the
properties ascribed to them by the theorems in the earlier sections.

\begin{figure}[ptb]
\centering
\includegraphics[width=8cm, keepaspectratio]{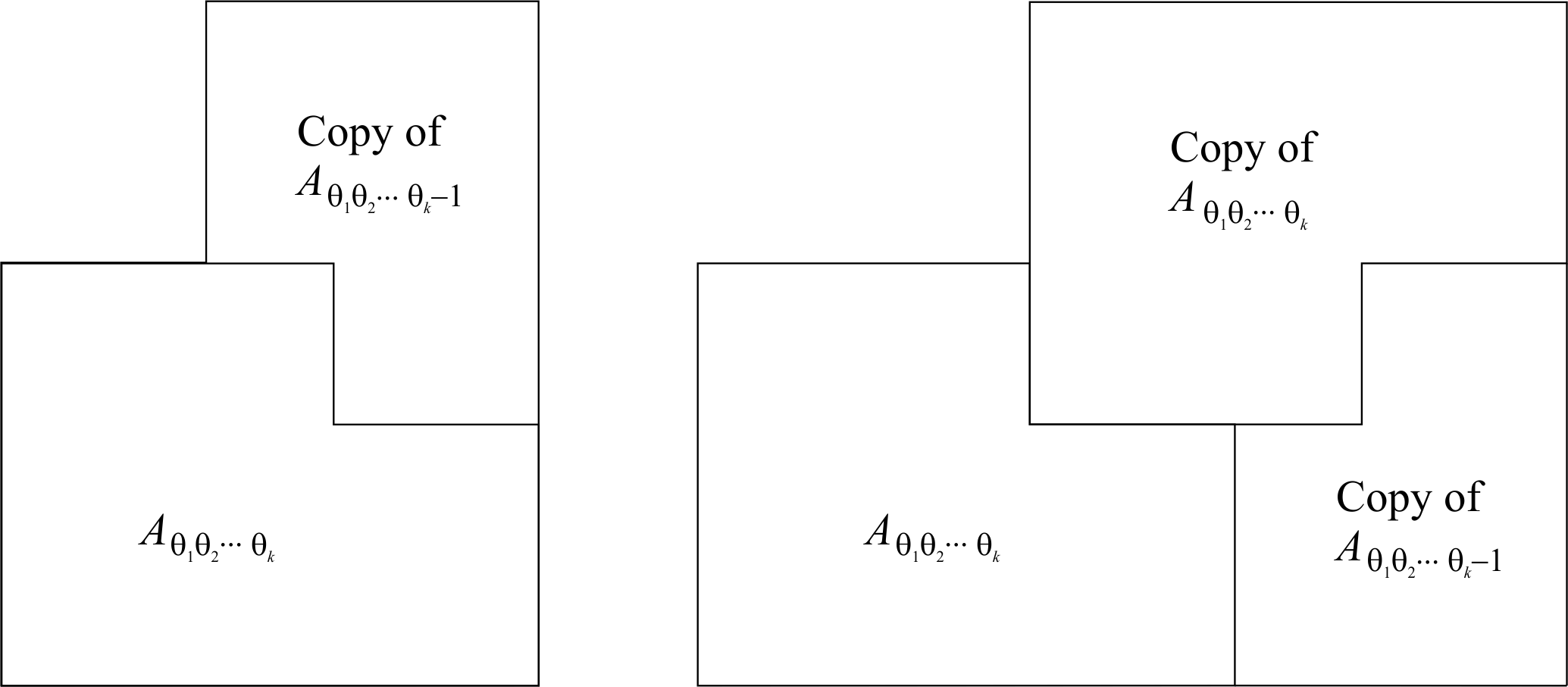} \caption{Structures of
$A_{\theta_{1}\theta_{2}\cdots\theta_{k}1}$ and $A_{\theta_{1}\theta_{2}%
\cdots\theta_{k}2}$ relative to $A_{\theta_{1}\theta_{2}\cdots\theta_{k}}$.}%
\label{construction}%
\end{figure}

\begin{figure}[ptb]
\centering
\vskip 7mm \includegraphics[width=8cm, keepaspectratio]{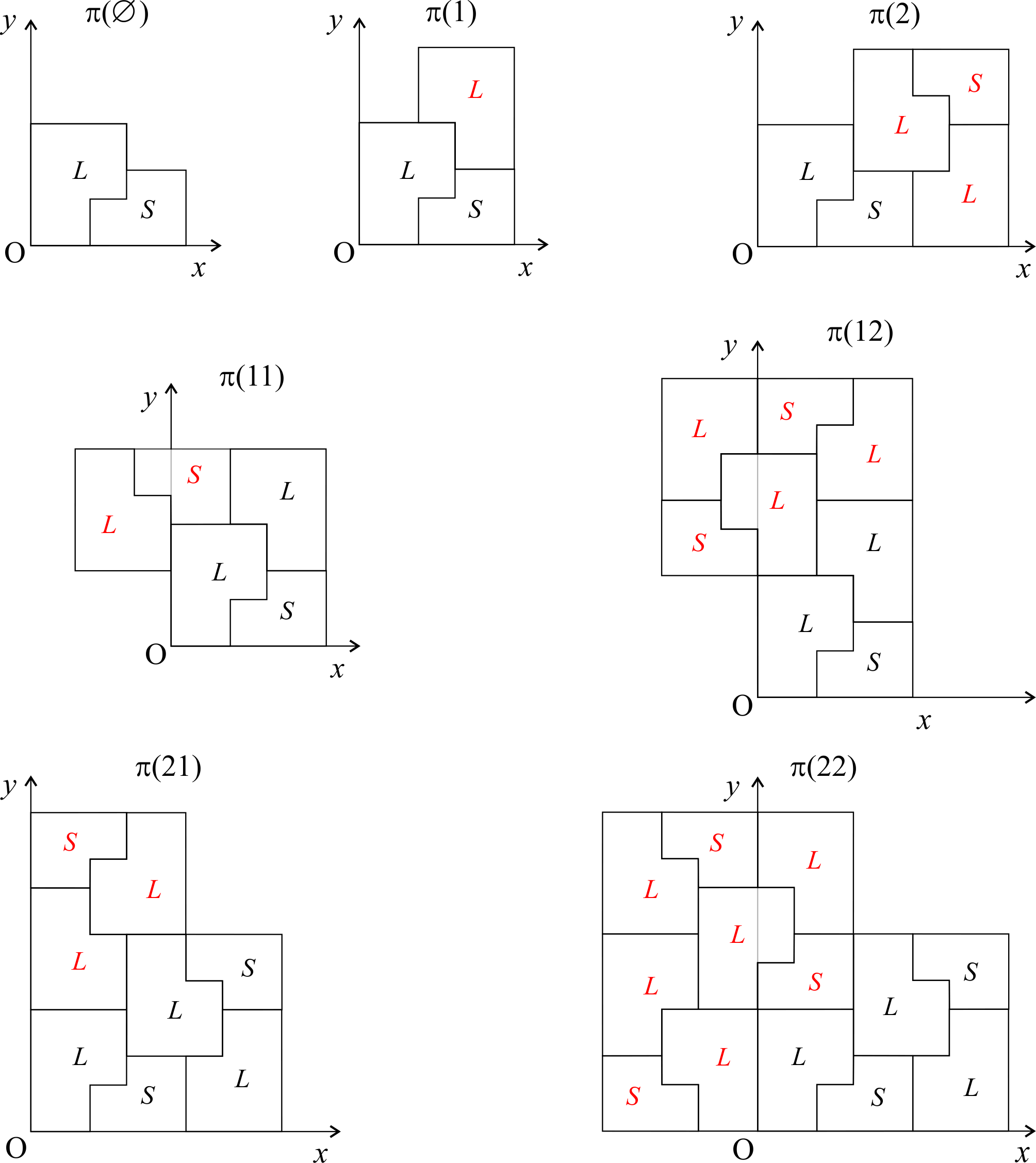}
\caption{Some of the sets $A_{\theta_{1}\theta_{2}\theta_{3}..\theta_{k}}$ and
the corresponding tilings $\pi(\theta_{1}\theta_{2}\theta_{3}..\theta_{k})$.
The recursive organization is such that $\pi(\varnothing)\subset\pi(\theta
_{1})\subset\pi(\theta_{1}\theta_{2})\subset\cdots$ regardless of the choice
$\theta_{1}\theta_{2}\theta_{3}..\in\{1,2\}^{\infty}$. }%
\label{img_1015x}%
\end{figure}

\begin{figure}[ptb]
\centering
\includegraphics[width=8cm, keepaspectratio
]{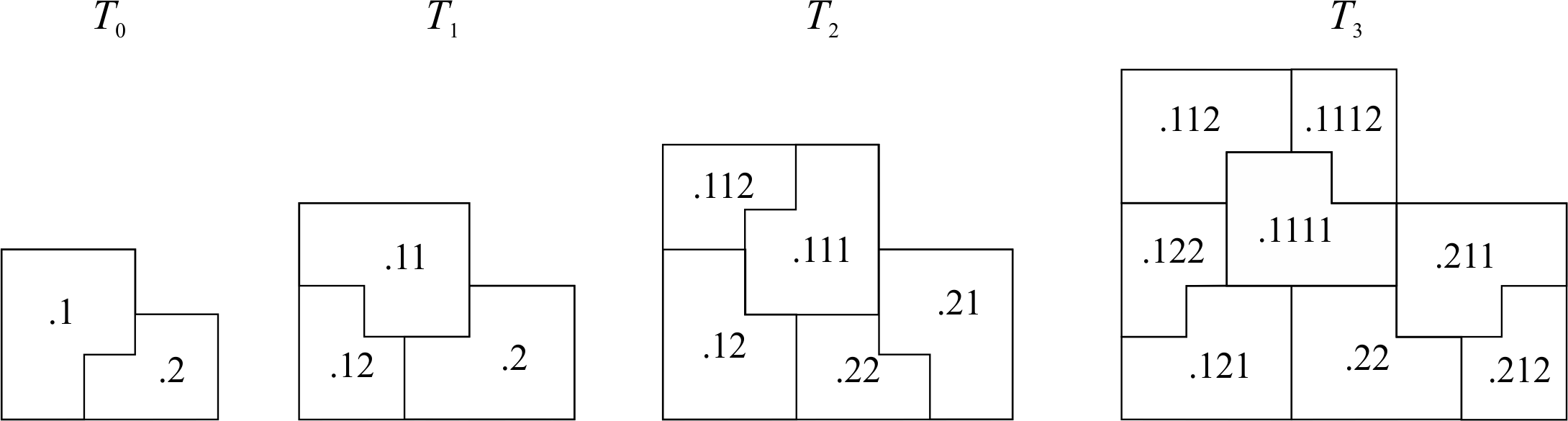}\caption{Relative addresses, the addresses of the tiles that
comprise the tilings $T_{0},T_{1},T_{2},T_{3}$ of $A_{0},A_{1},A_{2}%
,A_{3}\text{.}$}%
\label{l-maps}%
\end{figure}

\begin{figure}[ptb]
\centering
\vskip 9mm \includegraphics[width=8cm, keepaspectratio]{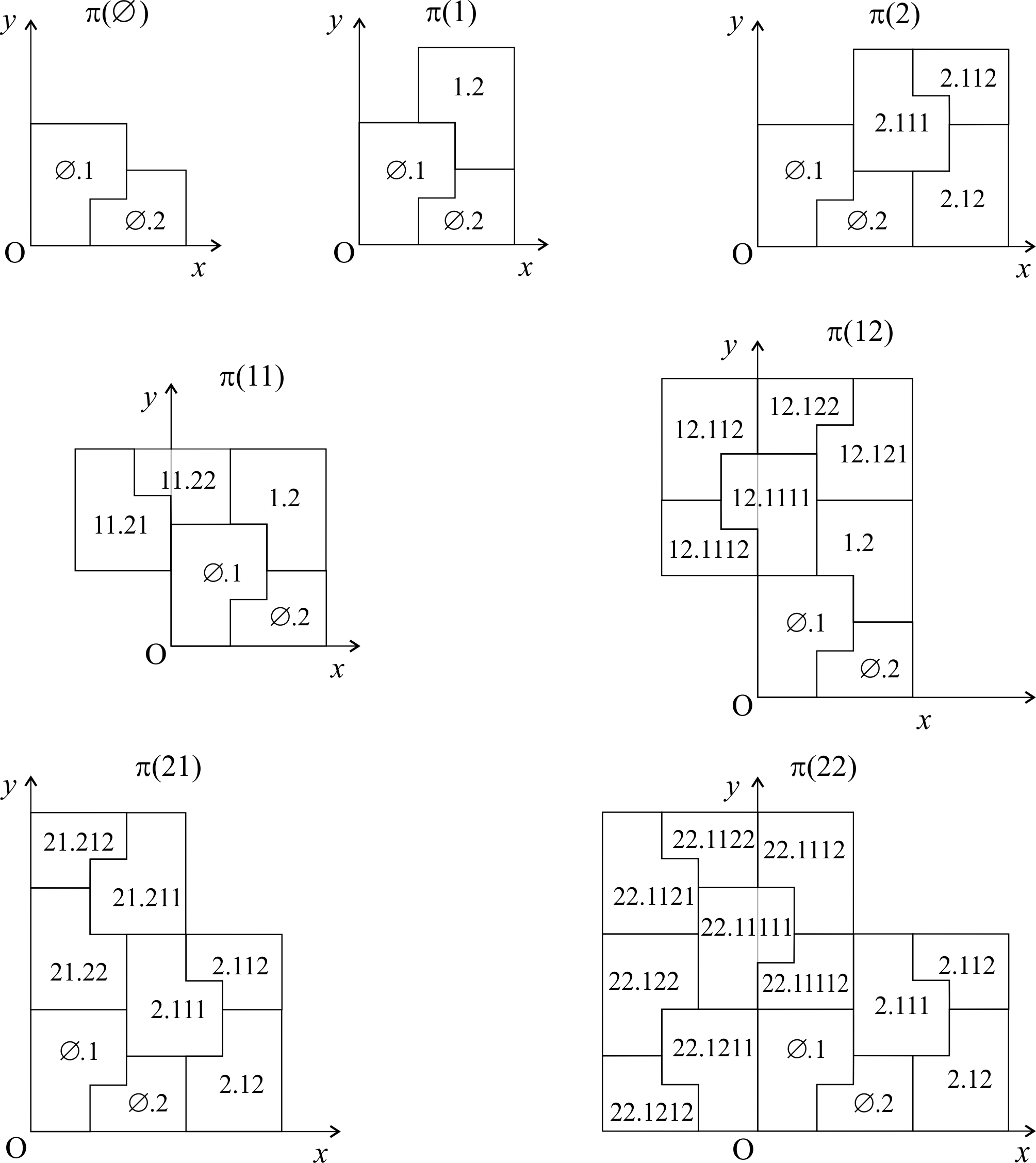}
\caption{Absolute addresses associated with the golden b.}%
\label{absolute}%
\end{figure}

\begin{figure}[ptb]
\centering
\vskip 9mm \includegraphics[width=8cm, keepaspectratio]{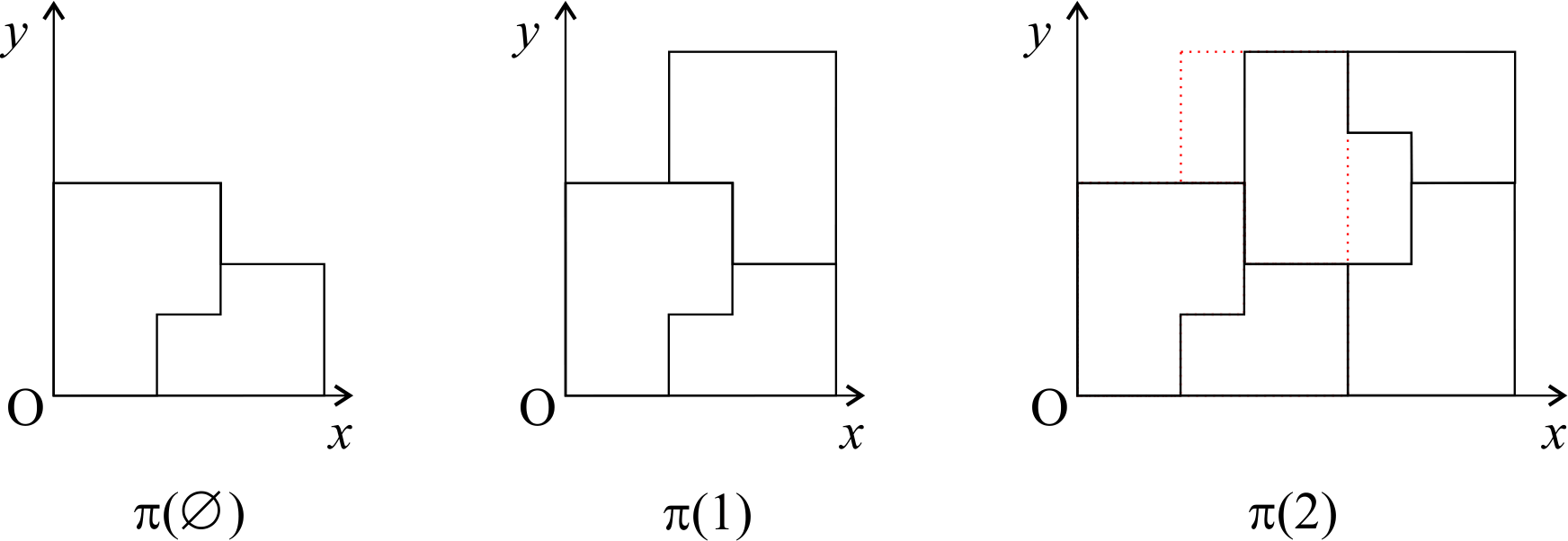}
\caption{The boundaries of the tilings $\pi(\emptyset)$, $\pi(1)$, $\pi(2)$,
with the parts of the boundaries of the tiles in $\pi(1)$ that are not parts
of the boundaries of tiles in $\pi(2)$ superimposed in red on the rightmost
image.}%
\label{img1013}%
\end{figure}

The relationships between $A_{\theta_{1}\theta_{2}\cdots\theta_{k}1}$ and
$A_{\theta_{1}\theta_{2}\cdots\theta_{k}2}$ relative to $A_{\theta_{1}%
\theta_{2}\cdots\theta_{k}}$ are illustrated in Figure \ref{construction}.
Figure \ref{img_1015x} illustrates some of the sets $A_{\theta_{1}\theta
_{2}\theta_{3}..\theta_{k}}$ and the corresponding tilings $\pi(\theta
_{1}\theta_{2}\theta_{3}..\theta_{k})$.

In Section \ref{realabsec}, procedures were described by which the relative
addresses of tiles in $T(\theta|k)$ and the absolute addresses of tiles in
$\pi(\theta|k)$ may be calculated recursively. Relative addresses for some
golden b tilings are illustrated in Figure \ref{l-maps}. Figure \ref{absolute}
illustrates absolute addresses for some golden b tilings.

The map $\pi:[2]^{\ast}\cup\lbrack2]^{\infty}\rightarrow\mathbb{T}$ is 1-1 by
Theorem \ref{1to1thm}, because $\pi(1)\cup\pi(2)$ does not tile the
interesection of the supports of $\pi(1)$ and $\pi(2),$ as illustrated in
Figure \ref{img1013}.

We note that $\pi(\overline{12})$ and $\pi(\overline{21})$ are aperiodic
tilings of the upper right quadrant of $\mathbb{R}^{2}$.

\subsection{Fractal tilings with non-integer dimension}

The left hand image in Figure \ref{gold8map}, shows the attractor of the IFS
represented by the different coloured regions, there being 8 maps, and
provides an example of a strongly rigid IFS. The right hand image represents
the attractor of the same IFS minus one of the maps, also strongly rigid, but
in this case the dimensions of the tiles is less than two and greater than
one. Figure \ref{sidebyside} (in Section~\ref{sec:intro}) illustrates a part
of a fractal blow up of a different but related 7 map IFS, also strongly
rigid, and the corresponding tiling.%

\begin{figure}[ptb]%
\centering
\includegraphics[
height=2.5728in,
width=5.2477in
]%
{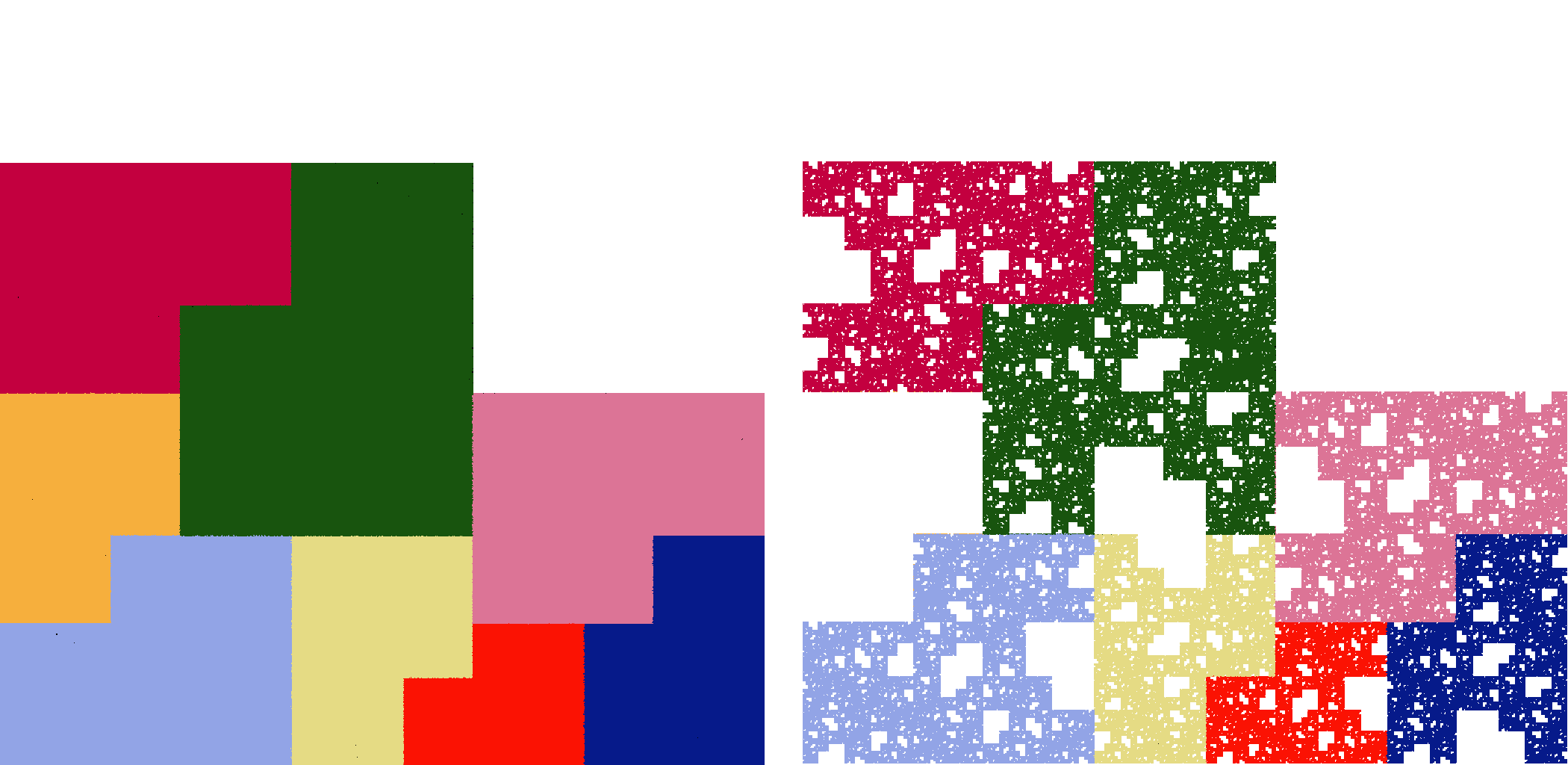}%
\caption{See text.}%
\label{gold8map}%
\end{figure}

Figure \ref{beetile} left shows a tiling associated with the IFS $\mathcal{F}$
represented on the left in Figure \ref{gold8map}, while the tiling on the
right is another example of a fractal tiling, obtained by dropping one of the
maps of $\mathcal{F}$.%
\begin{figure}[ptb]%
\centering
\includegraphics[
height=2.5728in,
width=5.2468in
]%
{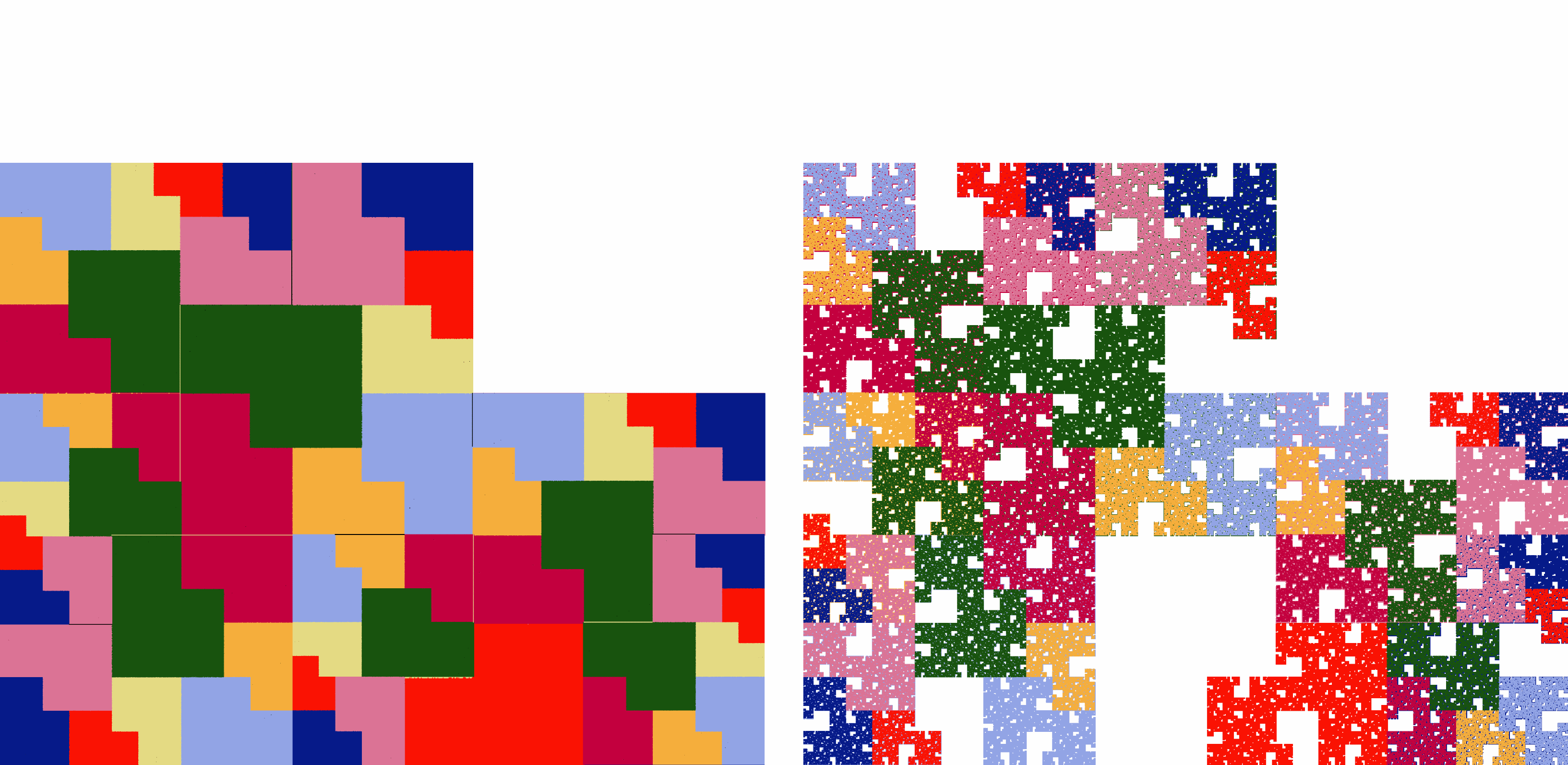}%
\caption{See text.}%
\label{beetile}%
\end{figure}

\subsection{Tilings derived from Cantor sets}

Our results apply to the case where $\mathcal{F}=\{\mathbb{R}^{M}%
;f_{i}(x)=s^{a_{i}}O_{i}+q_{i},i\in\lbrack N]\}$ where $\{O_{i},q_{i}%
:i\in\lbrack N]\}$ is fixed in a general position, the $a_{i}$s are positive
integers, and $s$ is chosen small enough to ensure that the attractor is a
Cantor set. In this situation the set of overlap is empty and it is to be
expected that $\mathcal{F}$ is strongly rigid, in which case all tilings (by a
finite set of prototiles, each a Cantor set) will be non-periodic. We can then
take $s$ to be the supremum of value such that the set of overlap is nonempty,
to yield interesting ``just touching" tilings.


\begin{thebibliography}{99}                                                                                               %


\bibitem {anderson}J. E. Anderson and I. F. Putnam, Topological invariants for
substitution tilings and their associated C$^{\ast}$-algebras, \textit{Ergod.
Th. \& Dynam. Sys. }\textbf{18 }(1998) 509-537.

\bibitem {bandt}C. Bandt, M. F. Barnsley, M. Hegland, A. Vince, Old wine in
fractal bottles I: Orthogonal expansions on self-referential spaces via
fractal transformations, \textit{Chaos, Solitons and Fractals, }\textbf{91}
(2016) 478-489.

\bibitem {manifold}M. F. Barnsley, A. Vince, Fast basins and branched fractal
manifolds of attractors of iterated function systems, \textit{SIGMA
}\textbf{11} (2015), 084, 21 pages.

\bibitem {tilings}M. F. Barnsley, A. Vince, Fractal tilings from iterated
function systems, \textit{Discrete and Computational Geometry, }\textbf{51}
(2014) 729-752.

\bibitem {polygon}M. F. Barnsley, A. Vince, Self-similar polygonal tilings,
\textit{Amer. Math. Monthly, }to appear (2017).

\bibitem {G}B. Gr\"{u}nbaum and G. S. Shephard, Tilings and Patterns, Freeman,
New York (1987).

\bibitem {hutchinson}J. Hutchinson, Fractals and self-similarity,
\textit{Indiana Univ. Math. J}. \textbf{30} (1981) 713-747.

\bibitem {Ken}R. Kenyon, The construction of self-similar tilings,
\textit{Geom. Funct. Anal.} \textbf{6} (1996) 471-488.

\bibitem {Pen}R. Penrose, Pentaplexity, \textit{Math Intelligencer}
\textbf{12} (1965) 247-248.

\bibitem {Rad}C. Radin, M. Wolff, Space tilings and local isomorphism,
\textit{Geometrica Dedicata }\textbf{42}, 355-360, 1991. 

\bibitem {S}K. Scherer, A Puzzling Journey To The Reptiles And Related
Animals, privately published, Auckland, New Zealand, 1987.

\bibitem {Sch}J. H. Schmerl, Dividing a polygon into two similar polygons,
\textit{Discrete Math.}, \textbf{311} (2011) 220-231.

\bibitem {sadun}L. Sadun, Tiling spaces are inverse limits, \textit{J. Math.
Phys., }\textbf{44} (2003) 5410-5414.

\bibitem {strichartz}R. S. Strichartz, Fractals in the large, \textit{Canad.
J. Math.}$,$\textit{ }\textbf{50 (}1998) 638-657.
\end{thebibliography}
\end{document}